\documentclass[a4paper, 10pt]{cmslatex}

\usepackage{amsmath, amssymb, amsfonts}
\usepackage{multirow}
\usepackage{booktabs}
\usepackage{appendix}
\usepackage[CJKbookmarks=true,colorlinks,linkcolor=black,anchorcolor=black,citecolor=black]{hyperref}
\usepackage{graphicx, xcolor}
\usepackage{tikz,pgfplots}
\usepackage{overpic,subfig, float}
\usepackage{listings, color}
\newcommand\+[1]{\boldsymbol{#1}}
\newcommand\tl[1]{\boldsymbol{\tilde{#1}}}
\newcommand\bbR{\mathbb{R}}
\newcommand\bbN{\mathbb{N}}

\newcommand\pd[2]{\dfrac{\partial {#1}}{\partial {#2}}}

\newcommand\odd[1]{\dfrac{\mathrm{d}}{\mathrm{d} {#1}}}
\newcommand\od[2]{\dfrac{\mathrm{d} {#1}}{\mathrm{d} {#2}}}

\newcommand\renew[1]{{#1}}

\newcommand\ang[1]{\left\langle{#1}\right\rangle}

\begin{document}

\title{Linear Moment Models to Approximate Knudsen Layers}

\author{Ruo Li\thanks{CAPT, LMAM \& School of Mathematical Sciences,
    Peking University, Beijing 100871, China, email: {\tt
      rli@math.pku.edu.cn}.} \and Yichen Yang\thanks{School of
    Mathematical Sciences, Peking University, Beijing 100871, China, email:
    {\tt yichenyang@pku.edu.cn}.}
}

\maketitle
\begin{abstract}
	We propose a well-posed Maxwell-type boundary condition for the
	linear moment system in half-space. As a reduction model of the
	Boltzmann equation, the moment equations are available to model
	Knudsen layers near a solid wall, where proper boundary conditions
	play a key role. In this paper, we will collect the
	moment system into the form of a general boundary value problem in 
	half-space. Utilizing an orthogonal decomposition, we separate 
	the part with a damping term from the system and then impose a 
	new class of Maxwell-type boundary conditions on it. Due to the
	block structure of boundary conditions, we show 
	that the half-space boundary value
	problem admits a unique solution with explicit expressions. 
	Instantly, the well-posedness of the linear moment system is 
	achieved. We apply the procedure to classical flow problems with 
	the Shakhov collision term, such as the velocity slip and 
	temperature jump problems. The model can capture Knudsen
	layers with very high accuracy using only a few moments.

\end{abstract}


\section{Introduction}

The Knudsen layer is an important rarefaction effect of gas flows
near the surface \cite{Lilley2007}, where non-Maxwellian velocity
distribution functions must be considered because of the gas-surface
interaction. The gas exhibits non-Newtonian
behavior in the Knudsen layer, and 
there is a finite velocity or temperature gap at the  
surface, known as the velocity slip or temperature
jump \cite{Kramers1949,Welander1954}. A better understanding of the
Knudsen layer may help design numerical methods for coupling the 
Boltzmann and Euler equations \cite{Chen2019,Gu2019}, 
avoiding solving the complex multidimensional Boltzmann equation
in the whole space.

The linearized Boltzmann equation (LBE) \cite{Williams2001} is 
widely used to depict the Knudsen layer. 
\renew{Half-space problems for the Boltzmann equation are}
often solved by the direct simulation Monte Carlo (DSMC)
method \cite{Bird} or discrete velocity/ordinates method (DVM/DOM) 
\cite{Broadwell,Siewert2001,Bern2010}. Numerical results 
for various collision models
have been reported  \cite{Loyalka1989,Loyalka1990, Siewert2001}. 
In theory, the well-posedness has been exhaustively studied
for linear half-space kinetic equations \cite{Bardos2006,Lu2017,Li2017}
and discrete Boltzmann equations in the DVM \cite{Bern2010}.

Meanwhile, Grad's moment method \cite{Grad1949} has been developed
\cite{Torrilhon2009,Fan_new,framework}
into a popular reduction model of the Boltzmann equation
with efficient numerical methods \cite{Cai2011,Rana2013,Hu2019}.
It is also available \cite{2008Linear,Gu2014,Lijun2017}
to model the Knudsen layer. Compared with kinetic equations,
the moment system often gives a formal analytical general solution
and leads to empirical formulas describing the gas behavior in the 
Knudsen layer. These formulas may help simplify the coupling of 
the Knudsen layer and bulk solutions. However, the 
Maxwell boundary condition proposed by Grad
\cite{Grad1949} is shown unstable
\cite{Sarna2018} even in the linearized case. For the linear 
initial-boundary value problem (IBVP) of moment equations,
\cite{Sarna2018} has defined stability criteria and constructed
the formulation of stable boundary conditions.

To our best knowledge, the well-posedness results are still scattered
for \renew{half-space problems} based on Grad's arbitrary order 
moment equations
with Maxwell-type boundary conditions. In numeric, there are also
few universal methods to deal with different flow problems. 
For example, \cite{Lijun2017} 
numerically solves Kramers' problem for the BGK \cite{BGK} model
and proves the well-posedness when the accommodation coefficient is an 
algebraic number. This paper aims to overcome these two lacks.

One of our main contributions is to propose a new class of 
Maxwell-type boundary conditions. 
It makes sure the well-posedness of the linear 
homogeneous moment system in half-space. The system
is derived from Grad's moment equations under some Knudsen layer
assumptions and can deal with different specific flow problems
with various collision terms. 
We first collect the system into a general half-space 
boundary value problem. Then we make an orthogonal decomposition
to separate the equations with a damping term from the whole system.
With the method of characteristics, we get several  
well-posedness criteria about boundary conditions. Under these
criteria, the solvability of the moment system is ensured. 
From the constructive proof, we can even write explicit analytical
solutions to the moment system. The procedure gives
an efficient algorithm to solve half-space problems,
which is another contribution of this paper.

Specifically, the construction of boundary conditions
mainly follows Grad's idea \cite{Grad1949,Cai2011} by imposing 
the continuity of odd fluxes \cite{1997Arnold} at the boundary. 
The obtained Maxwell-type boundary conditions are in a common 
even-odd parity form \cite{2008Linear,Lijun2017}. To meet 
the well-posedness criteria for half-space problems, the linear
space determined by Maxwell-type boundary conditions
is encouraged to contain the null space of the boundary matrix. 
This idea has been emphasized in many other problems
\cite{1960Local,Rauch1985,Sarna2018,Yong2019}. 
We will first get redundant boundary conditions and then
combine them linearly to meet the above criteria.

This paper is organized as follows. In Section 2, we summarize
the main well-posedness result of linear half-space
boundary value problems. In Section 3, we derive the moment
system in half-space with Maxwell-type boundary conditions.
In Section 4, we apply our model to velocity slip and 
temperature jump problems with the Shakhov collision term. 
The paper ends with a conclusion.

\section{Solvability Conditions for Half-Space Problems}

We consider the boundary layer problem arising in rarefied gas
flows when the Knudsen number tends to zero. These equations are
often linear \cite{Aoki2017} regardless of the nonlinear setting
of the original equations. Meanwhile, the resulting half-space 
moment system may have block structures due to the
orthogonality and the recursion relation of Hermite polynomials.

Therefore, we consider the linear half-space boundary 
value problem with constant coefficients
\begin{eqnarray}\label{eq:KL_0}
  \+A\od{\+w(y)}{y} &=&-\+Q\+w(y),\quad y\in[0,+\infty), \\
  \notag \+w(+\infty) &=& \+0,
\end{eqnarray}
where $\+w(y)\in\bbR^{m+n}$ with $m\geq n$, and
\begin{equation}\label{eq:bs}
  \+A=\begin{bmatrix} \+0 & \+M \\
    \+M^T & \+0\end{bmatrix}, \qquad
  \+Q=\begin{bmatrix}\+Q_e & \+0 \\
    \+0 & \+Q_o\end{bmatrix}.
\end{equation}
We assume that the matrix $\+M \in \bbR^{m\times n}$ has a full
column rank of $n$, and the matrices $\+Q_e\in\bbR^{m\times m}$ as
well as $\+Q_o\in\bbR^{n\times n}$ are symmetric positive 
semi-definite. We also let 
$\mathrm{Null}(\+A)\cap \mathrm{Null}(\+Q)=\{\+0\}$.

Our main point is to prescribe appropriate boundary conditions 
at $y=0$ to ensure the well-posedness of (\ref{eq:KL_0}).
The result is not as straightforward as it appears since the matrices
$\+A$ and $\+Q$ may both have zero eigenvalues 
\footnote{For equations like (\ref{eq:KL_0}), 
there are some classical results if $\+Q$ is symmetric
  positive definite \cite{1960Local} or $\+A$ does not have zero
  eigenvalues \cite{Hil2013}. Here we consider the situation
  without such restrictions. Incidentally, the well-posedness
  for kinetic layer equations is well solved
  \cite{Bardos2006}.}.

\renew{If $\+Q$ is invertible, we can consider the eigenvalue
decomposition of $\+Q^{-1}\+A$ and only prescribe boundary 
conditions for some ``incoming waves'' as in the hyperbolic 
problem. While both the coefficient matrices have zero
eigenvalues, we should consider the generalized eigenvalue 
problem of $\+A$ and $\+Q$. We try to reduce the generalized
eigenvalue problem by simultaneous reduction. Roughly 
speaking, we may split the space into several parts and only need
to prescribe boundary conditions for the projection of $\+w$ on
some spaces.

We first consider the null space of $\+Q$. We 
let $\+G\in\bbR^{(m+n)\times p}$ be the orthonormal basis 
matrix of $\mathrm{Null}(\+Q)$, i.e., $\+Q\+G=\+0$ and $\+G$ is column 
orthogonal. It's not enough to consider $\+Q$ and its
orthogonal complement. In fact, multiply (\ref{eq:KL_0}) left
by $\+G^T$ and we find 
$\+G^T\+A\+w(y)=0$ because $\+Q$ is symmetric and
$\+w(+\infty)=\+0.$ So the projection of $\+w$ on 
$\mathrm{span}\{\+A\+G\}$ does not need extra boundary
conditions. We need to consider the intersection of
$\mathrm{span}\{\+G\}$ and the orthogonal complement of
$\mathrm{span}\{\+A\+G\}$. 
Denote $\+V_1=\+G\+X\in\bbR^{(m+n)\times r}$
where $\+X\in\bbR^{p\times r}$ is the orthonormal basis
matrix of $\mathrm{Null}(\+G^T\+A\+G),$ i.e.,
$\+G^T\+A\+G\+X=\+0$ and $\+X$ is column orthogonal.
Then $\mathrm{span}\{\+V_1\}$ is
the intersection mentioned above.

To be rigorous, we let $\+U_1=\+G$,
$\+U_2=\+A\+G\+X$, and introduce the following 
lemma, whose proof is put in Appendix \ref{app:D}.
}

\begin{lemma}\label{lem:01}
  There exist matrices $\+V_2 \in \bbR^{(m+n)\times p}$ and
  $\+V_3 \in \bbR^{(m+n)\times (m+n-p-r)}$ such that  
  \begin{enumerate}
  \item $\mathrm{span}\{\+V_2\} =\mathrm{span}\{\+A\+G\}$;
	\item  $\+V=[\+V_1,\+V_2,\+V_3]$ is orthogonal and 
  $\+U=[\+U_1,\+U_2,\+V_3]$ is invertible;
  \item $\mathrm{rank}(\+U_2^T\+A\+V_1)=r$, $\+V_3^T\+A\+V_1 =\+0$
    and $\+V_3^T\+Q\+V_3$ symmetric positive definite.
  \end{enumerate}
\end{lemma}

\renew{The matrix $\+V_3$ satisfying Lemma \ref{lem:01} is not
unique, which can differ by an orthogonal transformation. However,
our proof would not rely on the choices of $\+V_3$ in this paper.}
As mentioned 
before, we may only prescribe additional boundary conditions at
$y=0$ for $\+V_3^T\+w(0)$.
So we assume the boundary condition 
\begin{equation}\label{eq:KL_wb}
\+B\+V_3 (\+V_3^T\+w(0)) 
	= \+g,\
\end{equation}
where the constant matrix $\+B\in\bbR^{n\times (m+n)}$ and 
the vector $\+g\in\bbR^n.$

We then denote by $\+Q_{33}=\+V_3^T\+Q\+V_3$ and
$\+A_{33}=\+V_3^T\+A\+V_3$. Let $n_+$ be the number of positive
eigenvalues of $\+Q_{33}^{-1}\+A_{33}$. We may show below that
there exists $\+T_+\in\bbR^{(m+n-p-r)\times n_+}$ and a positive
diagonal matrix $\+\Lambda_+$ such that
\begin{equation}\label{eq:T2}
  \+Q_{33}^{-1}\+A_{33}\+T_+=\+T_+\+\Lambda_+.
\end{equation}
One has the following well-posedness theorem.

\begin{theorem}\label{thm:01}
  The system \eqref{eq:KL_0} with the boundary condition
  \eqref{eq:KL_wb} has a unique solution of $\+w(y)$ if 
  the constant matrix $\+B$ and the vector $\+g$ in \eqref{eq:KL_wb}
	satisfy that 
  \begin{eqnarray}\label{eq:pf02}
    \mathrm{rank}(\+B\+V_3\+T_+) = n_+, \quad
    \+g\in\mathrm{span}\{\+B\+V_3\+T_+\},
  \end{eqnarray}
	\renew{where $\+V_3$ satisfies Lemma \ref{lem:01} and
		$\+T_+$ satisfies \eqref{eq:T2}.}
	What's more, the analytical solutions to the system are
	explicitly given by the expressions \eqref{eq:V1T}, \eqref{eq:V2T},
	and \eqref{eq:V3T} below.
\end{theorem}
\begin{proof}
  The proof is constructive. It is clear that \eqref{eq:KL_0} is
  equivalent to
  \begin{equation*}
    \+U^T\+A\+V\od{(\+V^T\+w)}{y}
    = - \+U^T\+Q\+V (\+V^T\+w),\quad
    \+w(+\infty)=\+0.
  \end{equation*}
  Since $\+G^T\+Q=\+0$ and $\+Q\+V_1=\+Q\+G\+X=\+0$, the above system
  becomes
  \begin{eqnarray*}
    \begin{bmatrix} * & * & * \\
      \+A_{21} & * & * \\
      \+A_{31} & * & \+A_{33}
    \end{bmatrix}\odd{y}
    \begin{bmatrix} \+V_1^T\+w\\
      \+V_2^T\+w \\ \+V_3^T\+w \end{bmatrix} = -
    \begin{bmatrix} \+0 & \+0 & \+0 \\ 
      \+0 & * & * \\
      \+0 & * & \+Q_{33}
    \end{bmatrix} \begin{bmatrix} \+V_1^T\+w\\
      \+V_2^T\+w \\ \+V_3^T\+w \end{bmatrix},\quad
    \+w(+\infty)=\+0,
  \end{eqnarray*}
  where $\+A_{21}=\+U_2^T\+A\+V_1$, $\+A_{31}=\+V_3^T\+A\+V_1$,
  $\+A_{33}=\+V_3^T\+A\+V_3$, $\+Q_{33}=\+V_3^T\+Q\+V_3$.

  With the condition $\+w(+\infty)=\+0,$ the first $p$ equations 
	would give
  \begin{equation}\label{eq:V2T}
    \+G^T\+A\+w(y) = \+0
    \quad\Rightarrow\quad \+V_2^T\+w(y)=\+0.
  \end{equation}
  By Lemma \ref{lem:01}, we have $\mathrm{rank}(\+A_{21})=r$,
  $\+A_{31}=\+0$ and $\+Q_{33}>0.$ So we can directly solve
  $\+V_1^T\+w$ from the next $r$ equations, if $\+V_3^T\+w(y)$ is
  known, to be
  \begin{eqnarray}\label{eq:V1T}
    \+V_1^T\+w(y) = - \+A_{21}^{-1}
    \+U_2^T\+A\+V_3(\+V_3^T \+w(y)) 
    + \int_{y}^{+\infty}\!\!
    \+A_{21}^{-1}\+U_2^T\+Q\+V_3(\+V_3^T \+w(s))
    \, \mathrm{d}s. 
  \end{eqnarray}
  Noting that $\+A_{31}=\+0,$ the last $m+n-p-r$ equations
  are separated alone as 
  \begin{equation*}
    \+A_{33}\odd{y}(\+V_3^T\+w) = -\+Q_{33}
    (\+V_3^T\+w),\ \+V_3^T\+w(+\infty)=\+0.
  \end{equation*}

	Since $\+Q_{33}>0,$ we apply the Cholesky decomposition 
	to have $\+Q_{33}=\+L\+L^T$. Then
	for the symmetric matrix $\+A_{33}$,
	due to Sylvester's law of inertia and the symmetry,
	there exists an orthogonal diagonalization
  $\+L^{-1} \+A_{33} \+L^{-T} \+R_+ = \+R_+\+\Lambda_+$ with
  $\+R_+^T\+R_+ = \+I_{n_+}$ where $\+\Lambda_+$ is a positive
	diagonal matrix. One possible choice of $\+T_+$ is 
  \begin{equation}\label{eq:T}
    \+T_+=\+L^{-T}\+R_+,
  \end{equation}
  then we have $\+Q_{33}^{-1}\+A_{33}\+T_+=\+T_+\+\Lambda_+.$
  
  Since $\+V_3^T\+w(+\infty)=0$, characteristic variables
  corresponding to non-positive eigenvalues of
  $\+Q_{33}^{-1} \+A_{33}$ have to be zero. Therefore, there exists
  $\+z\in\bbR^{n_+}$ such that $\+V_3^T\+w=\+T_+\+z$. Under the
  condition (\ref{eq:KL_wb})(\ref{eq:pf02}), the linear algebraic
  system
  \[ \+B\+V_3 (\+T_+\+z(0)) = \+g \]
  has a unique solution of $\+z(0)$. Consequently, we can uniquely
  solve 
  \begin{equation}\label{eq:V3T}
    \+V_3^T\+w(y) = \+T_+\exp\left(-\+\Lambda_+^{-1}y\right)\+z(0).
  \end{equation}

  Until now, we show that the unique solution of $\+w$ is determined
  by explicit expressions (\ref{eq:V1T}), (\ref{eq:V2T}), and
  (\ref{eq:V3T}).
\end{proof}

\renew{
	Theorem \ref{thm:01} has clarified the conditions to ensure the
	existence of a unique solution for the system (\ref{eq:KL_0}).
	Due to the explicit expressions (\ref{eq:V3T}), we can see that
	the solution $\+w$ continuously relies on $\+z(0)$, where $\+z(0)$
	is solved by a linear algebraic system with the boundary data
	$\+g$. Since all matrices have constant coefficients, the 
	solution must change continuously with the boundary data and
	Theorem \ref{thm:01} ensures the well-posedness of the 
	system (\ref{eq:KL_0}).
}

With the block structure in (\ref{eq:bs}), we can say more about the
value of $n_+$. The key point lies in the following lemma.
\begin{lemma}\label{lem:41}
Let $\+D \in \bbR^{\alpha \times \beta}$ and
$\mathrm{rank}(\+D)=\gamma$. Then $\begin{bmatrix} \+0 & \+D \\ \+D^T & \+0
\end{bmatrix}$ has $\alpha + \beta - 2\gamma$ zero eigenvalues,
$\gamma$ positive eigenvalues and $\gamma$ negative eigenvalues.
\end{lemma}
\begin{proof}
	The symmetric matrix $\tl{D}:=\begin{bmatrix} \+0 & \+D \\
	\+D^T & \+0 \end{bmatrix}$ must have $\alpha+\beta$ real
	eigenvalues. Assume $\lambda\in\bbR$ is an eigenvalue, 
	then there exists $\+x_e\in\bbR^{\alpha}$ and 
	$\+x_o\in\bbR^{\beta}$ such that 
	\[\begin{bmatrix} \+0 & \+D \\ \+D^T & \+0 \end{bmatrix}
		\begin{bmatrix} \+x_e \\ \+x_o \end{bmatrix} = \lambda
	\begin{bmatrix} \+x_e \\ \+x_o \end{bmatrix} \quad\Rightarrow
		\quad \begin{bmatrix} \+0 & \+D \\ \+D^T & \+0
	   	\end{bmatrix} \begin{bmatrix} \+x_e \\ -\+x_o \end{bmatrix}
	   	= -\lambda \begin{bmatrix} \+x_e \\ -\+x_o \end{bmatrix}.\]
		So $-\lambda$ is also an eigenvalue, which implies that 
		$\tl{D}$ has the same number of positive and negative 
		eigenvalues. Since $\mathrm{rank}(\+D)
			=\mathrm{rank}(\+D^T)=\gamma$, there must be
			$\alpha+\beta-2\gamma$ zero eigenvalues.
\end{proof}

{By some direct but lengthy manipulations, we can write
all matrices in Theorem \ref{thm:01} as a block form.}
Let $p_1,p_2\in\bbN$ that $p_1=\mathrm{dim}(\mathrm{Null}
(\+Q_e))$ and $p_2=\mathrm{dim}(\mathrm{Null}(\+Q_o))$. Then it's direct to
verify that $p=p_1+p_2$ and the matrices $\+G$ as well as $\+X$ can
be chosen as
\begin{equation}\label{eq:G2}
    \+G = \begin{bmatrix} \+G_e & \+0 \\ \+0 & \+G_o \end{bmatrix}, \qquad 
    \+X = \begin{bmatrix} \+X_e & \+0 \\ \+0 & \+X_o \end{bmatrix},
\end{equation}
where $\+G_e\in\bbR^{m\times p_1}$, $\+G_o\in\bbR^{n\times p_2}$,
$\+X_e\in\bbR^{p_1\times r_1}$, and $\+X_o\in\bbR^{p_2\times r_2}$.
Here we let $c=\mathrm{rank}(\+G_e^T\+M\+G_o)$ and $r_1=p_1-c$, 
$r_2=p_2-c$. By Lemma \ref{lem:41}, we have $r_1+r_2=r$ and
consequently 
	\[r_1+p_2=(p_1-c)+p_2=(p_2-c)+p_1=r_2+p_1=(r+p)/2.\]
	Denote by $\+Y_1=\+G_e\+X_e$ and $\+Z_1=\+G_o\+X_o.$
According to the proof of Lemma \ref{lem:01}, we can construct
the matrices $\+V_2$ and $\+V_3$ as
\begin{equation}\label{eq:V3}
    \+V_2=\begin{bmatrix} \+Y_2 & \+0 \\ \+0 & \+Z_2 \end{bmatrix},\quad
    \+V_3=\begin{bmatrix} \+Y_3 & \+0 \\ \+0 & \+Z_3 \end{bmatrix},
\end{equation}
where $\+Y_2\in\bbR^{m\times p_2}$ and $\+Z_2\in\bbR^{n\times p_1}$
that $\mathrm{span}\{\+Y_2\}=\mathrm{span}\{\+M\+G_o\}$,
    $\mathrm{span}\{\+Z_2\}=\mathrm{span}\{\+M^T\+G_e\}$.
$\+Y_3\in\bbR^{m\times (m-r_1-p_2)}$ and 
	$\+Z_3\in\bbR^{n\times (n-r_2-p_1)}$ are chosen to let
	$[\+Y_1,\+Y_2,\+Y_3]$ and $[\+Z_1,\+Z_2,\+Z_3]$ both be
	orthogonal. With these structures in mind, we note the 
	following lemma:

\begin{lemma}\label{lem:010}
    The matrix $\+Y_3^T\+M\+Z_3$ is of full column rank.
\end{lemma}

The proof of this lemma is put in Appendix \ref{app:D}.
With the help of this lemma,
\begin{corollary}\label{thm:02}
  The value of $n_+$ is $n_+=n-(p+r)/2$.
\end{corollary}
\begin{proof}
  According to (\ref{eq:V3}), we have
  $\+A_{33}=\begin{bmatrix} \+0 & \+Y_3^T\+M\+Z_3 \\ \+Z_3^T\+M^T\+Y_3
    & \+0 \end{bmatrix}$.  Since $\+Y_3^T\+M\+Z_3$ is of full column
  rank, the matrix $\+A_{33}$ should have $n-(p+r)/2$
  positive eigenvalues by Lemma \ref{lem:41}. 
\end{proof}


\section{Half-Space Moment System}
\subsection{Knudsen Layer Equations.}\label{sec:31}
We focus on Grad's arbitrary order moment equations \cite{Grad1949}
which are about moments of the velocity distribution function. 
Besides low-order moments such as the density, the macroscopic 
velocity, the temperature, etc., Grad tested the Boltzmann 
equation by multidimensional Hermite 
polynomials to get equations of higher-order moments.
For single-species monatomic gas without external forces,
{the linear homogeneous half-space problem for Grad's moment
equations}
(cf. Appendix \ref{app:F}) can write as 
\begin{eqnarray}\label{eq:KL_pre}
	\+A_2\od{\+h(y)}{y} &=& -\+J\+h(y),\quad y\in[0,+\infty),\\
	\quad \+h(+\infty)&=&\+0,\notag
\end{eqnarray}
where $\+h=\+h(y)\in\bbR^N$ and $N=\#\mathbb{I}_M$ with 
$\mathbb{I}_M=\{\+\alpha\in\bbN^D:\ |\+\alpha|\leq M\}.$
Here $\+\alpha=(\alpha_1,...,\alpha_D)\in\bbN^D$ is a multi-index
with $|\+\alpha|=\sum \alpha_i$, and $M$ is often called
the moment order.

For economy of words, we use $\+h[\+\alpha]$ to represent its
$\mathcal{N}(\+\alpha)$-th component, analogously for other 
vectors and matrices, where the one to one mapping
$\mathcal{N}:\ \mathbb{I}_M\rightarrow\{1,2,...,N\}$ is defined
as follows.

\begin{defn}\label{def:order}
For $\+\alpha,\+\beta\in\mathbb{I}_M$, 
\begin{enumerate}
	\item If $\alpha_2$ is even and $\beta_2$ is odd, then 
		$\mathcal{N}(\+\alpha)<\mathcal{N}(\+\beta)$.
	\item If $\alpha_2$ and $\beta_2$ have the same parity, but
		$|\+\alpha|<|\+\beta|$, then $\mathcal{N}(\+\alpha)<
		\mathcal{N}(\+\beta)$.
	\item If $\alpha_2$ and $\beta_2$ have the same parity and
		$|\+\alpha|=|\+\beta|$, but there exists a smallest 
		$1\leq i\leq D$ such that $\alpha_i\neq\beta_i$. Then
		$\mathcal{N}(\+\alpha)<\mathcal{N}(\+\beta)$ when
		$\alpha_i>\beta_i$. Otherwise $\mathcal{N}(\+\alpha)>
		\mathcal{N}(\+\beta)$.
	\end{enumerate}
\end{defn}
\renew{In the indexing provided by Defintion 3.1, the indices with even 
second components are always ordered before the ones with odd second
components, e.g. $(a_1,0,a_3)$ is before $(b_1,1,b_3)$. Then the 
indices are sorted by the multi-index norm and finally by the
anti-lexicographic order. For example, indices from
$\{\+\alpha\in\bbN^3,\ |\+\alpha|\leq 2\}$ are sorted as
$(0,0,0)$,$(1,0,0)$,$(0,0,1)$,$(2,0,0)$,$(1,0,1)$,$(0,2,0)$,
$(0,0,2)$, $(0,1,0)$,$(1,1,0)$,$(0,1,1).$
}

Now the coefficient matrices in (\ref{eq:KL_pre}) have
the following explicit expressions:
\begin{gather} \label{eq:LHME_A}
	\+A_2[\+\alpha,\+\beta] = \ang{\xi_2\omega \phi_{\+\alpha}
	\phi_{\+\beta}} = \sqrt{\alpha_2}\delta_{\+\beta,\+\alpha-\+e_2}
		+\sqrt{\alpha_2+1}\delta_{\+\beta,\+\alpha+\+e_2},\\
\label{eq:Q} \renew{\+J[\+\alpha,\+\beta] = -
	\ang{\mathcal{L}[\omega\phi_{\+\beta}]\phi_{\+\alpha}}},
\end{gather}
where $\delta_{\+\alpha,\+\beta}$ equals one when $\+\alpha=\+\beta$
and otherwise zero. Here $\+e_2\in\bbN^D$ is a unit vector with
only the second component being one.

\begin{defn}\label{def:Hermite}
In the above expressions, we denote the integral on the whole velocity
space by
	\[\ang{\cdot}=\displaystyle\int_{\bbR^D}\!\!
\cdot\,\mathrm{d}\+\xi.\]
	The isotropic weight function $\omega$ and the orthonormal
	Hermite polynomial $\phi_{\+\alpha}$ are defined as 
\begin{eqnarray*}
	\omega = \omega(\+\xi) &=& {(2\pi)^{-D/2}}\exp\left(
	-{|\boldsymbol{\xi}|^2}/{2}\right),\\
	\phi_{\+\alpha} = \phi_{\+\alpha}(\+\xi) &=&
	\frac{(-1)^{|\boldsymbol{\alpha}|}}{\sqrt{\+\alpha !}}
	\dfrac{\partial^{|\boldsymbol{\alpha}|} 
		\omega}{\partial \boldsymbol{\xi}^{\boldsymbol{\alpha}}}
		\omega^{-1},
\end{eqnarray*}
	where $\+\xi=(\xi_1,...,\xi_D)\in\bbR^D,\ \+\xi^{\boldsymbol{\alpha}}
	=\prod\xi_i^{\alpha_i},\ \+\alpha!=\prod \alpha_i!$. 
	According to \cite{Fan_new,Grad1949N}, we have
	\begin{itemize}
	\item Recursion relation. 
		\[\xi_d\phi_{\+\alpha} = 
		\sqrt{\alpha_d}\phi_{\+\alpha-\+e_d}+
		\sqrt{\alpha_d+1}\phi_{\+\alpha+\+e_d},\quad d=1,2,...,D.\]
	\item Orthogonal relation.
		\[ \ang{\omega\phi_{\+\alpha}\phi_{\+\beta}}=\delta_{\+\alpha,
			\+\beta}.\]
	\end{itemize}
	The linearized Boltzmann operator $\mathcal{L}[f](\+\xi)$ is 
	defined as
\[ \mathcal{L}[f](\+\xi) = \int_{\bbR^D}\!\!
	\int_{\mathbb{S}^{D-1}}\!\!
\mathcal{K}[f/\omega]
\omega(\+\xi)\omega(\+\xi_1)B(|\+\xi-\+\xi_1|,\Theta)
\,\mathrm{d}\Theta\mathrm{d}\+\xi_1,\]
where $B(|\+\xi-\+\xi_1|,\Theta)$ is a collision kernel.
The operator $\mathcal{K}$ is defined as 
\[ \mathcal{K}[g](\+\xi,\+\xi_1,\Theta)=g(\+\xi')+g(\+\xi_1')-
g(\+\xi)-g(\+\xi_1),\]
where the post-collisional velocities $\+\xi'$ and $\+\xi_1'$ are 
determined by $\+\xi,\ \+\xi_1$ as well as $\Theta$.
\end{defn}

Calculations of $\+J$ can refer to \cite{Wang2019}. We also consider
the (linearized) Shakhov model \cite{Shakhov}, which aims to approximate
the original (linearized) Boltzmann equation. It's defined by setting
$\+J=\+J_{Sh}$ in (\ref{eq:KL_pre}), where nonzero entries of $\+J_{Sh}$
(cf. \cite{Cai2011}) are
\begin{eqnarray}\label{eq:Q_Sh}
\+J_{Sh}[2\+e_i,2\+e_j] &=& \delta_{ij}-\frac{1}{D},\ i,j=1,2,...,D;\\
\+J_{Sh}[\+e_i+2\+e_j,\+e_i+2\+e_k] &=& \delta_{jk}-  \notag
\frac{1-\Pr}{5}\sqrt{1+2\delta_{ij}}\sqrt{1+2\delta_{ik}},\
i,j,k=1,2,...,D;\\ \notag
\+J_{Sh}[\+\alpha,\+\alpha] &=& 1,\ |\+\alpha|\geq 2 \text{\ and\ }
\+\alpha\neq 2\+e_i,\ \+e_i+2\+e_j.
\end{eqnarray}
When the Prandtl number $\Pr=1,\ $ the Shakhov model reduces to the
celebrated BGK model. 

Considering the computational efficiency, (\ref{eq:KL_pre}) is
unsatisfactory \renew{to solve half-space problems} since 
the size of (\ref{eq:KL_pre}) reaches
\[
N={\#}\{\+\alpha\in\bbN^D:\ |\+\alpha|\leq M\} = \binom{M+D}{D}
= \frac{(M+D)\cdots(M+1)}{D!} = O(M^D),
\]
but very few quantities are cared about in physics.
We can formally write the density $\rho_K$, the macroscopic velocity
$u_{i,K}$, the temperature $\theta_K$ and the stress tensor
$\sigma_{ij,K}$ as 
\begin{equation}\label{eq:LHME_h}
\rho_K =\+h[\+0],\ u_{i,K} =\+h[\+e_i],\ \frac{\sigma_{ij,K}
+\delta_{ij}\theta_K}{\sqrt{1+\delta_{ij}}} = 
\+h[\+e_i+\+e_j],
\end{equation}
where we assume $\sum \sigma_{ii,K}=0$ and the subscript $K$ is
attached to discriminate the Knudsen layer solutions.
\renew{Then, for example, in Kramers' problem, we consider
the tangential velocity $u_{1,K}$, where a gas flow passes by an 
infinite plate.  We assume the gas velocity is parallel 
to the plate, and the only driven force is from the tangential stress.
Due to the symmetry, Kramers' problem is essentially a
one-dimensional problem rather than the $D$-dimensional problem.
For the BGK model, the moment system (\ref{eq:KL_pre}) for
Kramers' problem \cite{Lijun2017} can reduce to
\begin{gather*}
	\od{\sigma_{12,K}}{y} = 0,\\
	\od{u_{1,K}}{y} + \sqrt{2}\od{\+h[\+e_1+2\+e_2]}{y}
	= -\sigma_{12,K},\\
	\sqrt{2}\od{\sigma_{12,K}}{y} + \sqrt{3}
	\od{\+h[\+e_1+3\+e_2]}{y} = -\+h[\+e_1+2\+e_2],\\
	\cdots\\
	\sqrt{k}\od{\+h[\+e_1+(k-1)\+e_2]}{y} + 
	(1-\delta_{k,M-1})\sqrt{k+1}
	\od{\+h[\+e_1+(k+1)\+e_2]}{y} = -\+h[\+e_1+k\+e_2],
	\quad k\leq M-1.
\end{gather*}
When $M=4$, the small moment system is
\begin{gather*}
	\od{\sigma_{12,K}}{y} = 0,\\
	\od{u_{1,K}}{y} + \sqrt{2}\od{\+h[\+e_1+2\+e_2]}{y}
	= -\sigma_{12,K},\\
	\sqrt{2}\od{\sigma_{12,K}}{y} + \sqrt{3}
	\od{\+h[\+e_1+3\+e_2]}{y} = -\+h[\+e_1+2\+e_2],\\
	\sqrt{3}\od{\+h[\+e_1+2\+e_2]}{y} = -\+h[\+e_1+3\+e_2].
\end{gather*}	
}

This inspires us to reduce (\ref{eq:KL_pre}) for specific
flow problems. 
\renew{We introduce a projection matrix 
$\+P_{\mathbb{I}}\in\bbR^{N\times(m+n)}$ relying on an index
set $\mathbb{I}\subset\mathbb{I}_M$}, and the reduced system would be
\begin{equation}\label{eq:KL_p2}
	\+A\od{\+w}{y} = -\+Q\+w,\ \+w(+\infty)=\+0;
	\quad 
	\+w=\+P_{\mathbb{I}}^T\+h,\ \+A=\+P_{\mathbb{I}}^T\+A_2
	\+P_{\mathbb{I}},\ \+Q=\+P_{\mathbb{I}}^T\+J\+P_{\mathbb{I}}. 
\end{equation}
\renew{Now (\ref{eq:KL_p2}) consists of $m+n$ equations, and we 
expect to reduce the system without compromising the accuracy 
of the concerned physical quantities.
Assume $\+P_{\mathbb{I}}$ is column orthogonal. We suggest 
choosing $\+P_{\mathbb{I}}$ such that it spans an
invariant space of $\+A_2$ and $\+J$, i.e.,
$\+A_2\+P_{\mathbb{I}}=\+P_{\mathbb{I}}\+C_1,\ \+J\+P_{\mathbb{I}}
=\+P_{\mathbb{I}}\+C_2$ for some matrices $\+C_1,\+C_2.$ If so,
the system (\ref{eq:KL_p2}) can be obtained by multiplying 
(\ref{eq:KL_pre}) left by $\+P_{\mathbb{I}}^T.$}

In this paper, we let $\+P_{\mathbb{I}}$ be the selection 
matrix with $\+P_{\mathbb{I}}[\+\alpha,\+\beta]=
\delta_{\+\alpha,\+\beta},$ where $\+\alpha\in\mathbb{I}_M$ and
$\+\beta\in\mathbb{I}.$ Thus (\ref{eq:KL_p2}) can be viewed as 
selecting some row equations of (\ref{eq:KL_pre}) and then 
dropping out extra unknowns to close the system.
Here $\+P_{\mathbb{I}}[\+\alpha,\+\beta]$ means the entry in the 
$\mathcal{N}(\+\alpha)$-th row and $\mathcal{N}_1(\+\beta)$-th
column where $\mathcal{N}$ is defined in Definition \ref{def:order}
and $\mathcal{N}_1:\ \mathbb{I}\rightarrow\{1,2,...,m+n\}$ gives
an indexing of $\mathbb{I}$. We define $\mathcal{N}_1$ by retaining
the order in $\mathbb{I}_M$, i.e.,
\[ \mathcal{N}_1(\+\alpha) < \mathcal{N}_1(\+\beta) 
\ \Leftrightarrow\ \mathcal{N}(\+\alpha) < \mathcal{N}(\+\beta),
\quad \+\alpha,\+\beta\in\mathbb{I}\subset\mathbb{I}_M.\]
Hence, if we assume  
\begin{eqnarray*}
	\mathbb{I}_e = \{\+\alpha\in\mathbb{I},\ \alpha_2 \text{\ even}\},
	\quad \mathbb{I}_o = \{\+\alpha\in\mathbb{I},\ \alpha_2
	\text{\ odd}\},\quad m=\#\mathbb{I}_e,\ n=\#\mathbb{I}_o,
\end{eqnarray*}
then $\+\alpha\in\mathbb{I}_e$ is always ordered in front of
$\+\beta\in\mathbb{I}_o$.

Due to the special sparsity pattern (\ref{eq:LHME_A}) of $\+A_2$,
\renew{it can be checked that $\+P_{\mathbb{I}}$ with the following
choice of $\mathbb{I}$ spans an invariant space of $\+A_2$}:
\begin{eqnarray*}
	{\bf{(C1).}}\quad \text{\ If\ } \+\gamma=(\gamma_i)\in\mathbb{I},
	\text{\ then\ } \+\gamma-\gamma_2\+e_2+\alpha_2\+e_2
	\in\mathbb{I},\ 0\leq \alpha_2\leq M-|\+\gamma|+\gamma_2.
\end{eqnarray*}
\renew{Intuitively, this condition says that all indices differing
only by the second component should either belong to or not belong
to $\mathbb{I}$. For example, if $(1,0,0)\in\mathbb{I}$, then 
$(1,1,0)$,$(1,2,0)$,...,$(1,M-1,0)$ should all belong to $\mathbb{I}$.}
Under this choice, every $\+\alpha\in\mathbb{I}_o$ implies
$\+\alpha-\+e_2\in\mathbb{I}_e$, which leads to $m\geq n$.
Thus, from (\ref{eq:LHME_A}), we immediately know that 
$\+A$ has the block structure
	\begin{equation}\label{eq:MM}
		\+A=\begin{bmatrix} \+0 & \+M \\ \+M^T & \+0\end{bmatrix},
			\quad 
		\+M[\+\alpha,\+\beta] 
		= \ang{\xi_2\omega
		\phi_{\+\beta}\phi_{\+\alpha}}
		,\quad \+\alpha\in\mathbb{I}_e,\ \+\beta\in\mathbb{I}_o.
	\end{equation}
As proved in Appendix \ref{app:D}, the equations (\ref{eq:KL_p2})
in fact satisfy the following lemma:
\begin{lemma}\label{lem:31}
	Assume $\mathbb{I}$ satisfies {\bf (C1)} and the multi-indices
	are ordered by Definition \ref{def:order}. Then, Knudsen layer 
	equations \eqref{eq:KL_p2} satisfy conditions in \eqref{eq:KL_0},
	i.e., 
	\begin{enumerate}
		\item The matrix $\+A$ has a block structure as \eqref{eq:MM},
			where $\mathrm{rank}(\+M)=n.$
		\item The matrix $\+Q$ has a block structure as \eqref{eq:bs},
			and is symmetric positive semi-definite.
		\item $\mathrm{Null}(\+A)\cap \mathrm{Null}(\+Q)=\{\+0\}$.
	\end{enumerate}
	The lemma also holds for the Shakhov model, i.e. $\+J=\+J_{Sh}$ in
	\eqref{eq:KL_p2}.
\end{lemma}

\begin{remark}\label{rem:01}
	From another point of view, \eqref{eq:KL_p2} can be viewed as
	directly testing the kinetic layer equations by linear 
	combinations of Hermite polynomials, and then 
	truncating the collision term to obtain a closed system.
	It's clearly shown in the expressions
	\eqref{eq:LHME_A}\eqref{eq:Q}.
\end{remark}	

\begin{remark}
The rigorous error analysis between \eqref{eq:KL_p2} and kinetic
layer equations is beyond this paper's scope. We notice that if
$\+J$ shares a sparsity pattern, then 
\renew{$\+P_{\mathbb{I}}$ can be chosen to span an invariant space 
	of $\+J$ easily.}	
	For example, if we consider the Shakhov model and focus on
	the tangential velocity $\+h[\+e_1]$, we can choose
\begin{equation*}
	\mathbb{I}=\{\+e_1+2\+e_j+\alpha_2\+e_2,\ j\neq 2,\ 0\leq\alpha_2
\leq M-3\}\cup\{\+e_1+\alpha_2\+e_2,\ 0\leq\alpha_2\leq M-1\}.
\end{equation*}
\renew{When $D=3$, the above $\mathbb{I}$ consists of 
	all indices in the form of $(1,\alpha_2,0)$, $(3,\alpha_2,0)$ 
	and $(1,\alpha_2,2)$. So $\#\mathbb{I}=3M-4$ and} 
the scale is reduced from $O(M^D)$ to $O(M)$.
\renew{It's direct to check that
	$\+J_{Sh}\+P_{\mathbb{I}}=\+P_{\mathbb{I}}\+C$ for some matrix
	$\+C$.}
\end{remark}

\subsection{Maxwell-Type Boundary Conditions.}\label{sec:32}
Maxwell's boundary condition describes the diffuse-specular process 
between the gas and solid wall, i.e., the reflected distribution of
particles is divided into a sum of $\chi$ portion of diffuse reflection
and $(1-\chi)$ portion of specular reflection:
\begin{equation}\label{eq:Maxwell}
	f(\+\xi) = \chi\mathcal{M}^w(\+\xi) + (1-\chi)f(\+\xi^*),\quad
	(\+\xi-\+u^w)\cdot\+n<0,
\end{equation}
where $\chi\in[0,1]$ is the accommodation coefficient. 
The wall is assumed impermeable and 
can not deform. We assume the reflected distribution caused by
the diffuse reflection is the Maxwellian
\begin{equation}
\mathcal{M}^w(\+\xi) = \frac{\rho^w}{(2\pi\theta^w)^{D/2}}
\exp\left(-\frac{|\+\xi-\+u^w|^2}{2\theta^w}\right),\ 
\end{equation}
where $\+u^w$ and $\theta^w$ are given, with $\rho^w$ determined by 
the no mass flow condition $(\+u-\+u^w)\cdot\+n=0.$ Here $\+n$ is the
outward unit normal vector at the boundary and $\+\xi^*=\+\xi-2\+n
(\+\xi\cdot\+n)$. In Knudsen layer problems, 
we further assume $\+n=(0,-1,0,...,0)$ and $\+u^w\cdot\+n=0$.

Following Grad's idea \cite{Grad1949}, we test (\ref{eq:Maxwell})
by polynomials which are odd about the direction normal to the
boundary. Rather than directly choosing \cite{Cai2011,Sarna2018}
the Hermite polynomials, we let the test polynomials be
\begin{equation*}
p_{\boldsymbol{\alpha}}(\boldsymbol{\xi}) 
	= \xi_2\phi_{\+\alpha}(\boldsymbol{\xi}),\
\boldsymbol{\alpha} \in\mathbb{I}_e,
\end{equation*}
which may simplify the analysis. Then under some Knudsen layer 
assumptions, when $\chi$ is not a small quantity, we would 
(cf. Appendix \ref{app:C}) have $m$ boundary conditions in an 
abstract form:
\begin{equation}\label{eq:wbc1}
	\+M(\+w_o(0)+\+f_o) = -b(\chi)\+S(\+w_e(0) + \+f_e),\ 
	\+w=\begin{bmatrix}\+w_e\\ \+w_o\end{bmatrix},\ 
	\+w_e\in\bbR^m,\ \+w_o\in\bbR^n,
\end{equation}
where $b(\chi)=\frac{2\chi}{2-\chi}\left(\sqrt{2\pi}\right)^{-1}$
and $\+f_o\in\bbR^n,\ \+f_e\in\bbR^m$. Here
$\+w$ is the unknown in (\ref{eq:KL_p2}) and $\+M$ is given 
in (\ref{eq:MM}). Entries of $\+S\in\bbR^{m\times m}$ are 
\begin{equation}\label{eq:defS}
	\+S[\+\alpha,\+\beta]= 
	\frac{\sqrt{2\pi}}{2} \ang{|\xi_2|\omega
	\phi_{\+\alpha}\phi_{\+\beta}},
	\quad \+\alpha,\+\beta\in\mathbb{I}_e,
\end{equation}
which are explicitly calculated in Appendix \ref{app:B}. What's more,
Appendix \ref{app:D} will show that 

\begin{lemma}\label{lem:03}
	The matrix $\+S$ defined in \eqref{eq:defS} is symmetric positive
	definite.
\end{lemma}
\begin{remark}
Grad's boundary condition can be regarded as
	\begin{equation}\label{eq:obc}
		\+E\+M(\+w_o(0)+\+f_o) = -b(\chi)\+E\+S(\+w_e(0) + \+f_e), 
	\end{equation}
where $\+E=[\+I_n,\+0]$ with $\+E\in\bbR^{n\times m}$. In the
	IBVP for moment equations, this gives a correct number
\cite{Grad1949,Rauch1985} of boundary conditions for the linear
hyperbolic system. But the well-posedness of \eqref{eq:obc} has
	never been proved. In fact, \cite{Sarna2018} shows that
	\eqref{eq:obc} is unstable in the IBVP. A similar problem 
	arises in the half-space problem, where it's difficult to 
	prove the well-posedness of Grad's boundary condition in 
	the general case. 
\end{remark}

By Lemma \ref{lem:31}, the Knudsen layer equations (\ref{eq:KL_p2})
satisfy all conditions in Theorem \ref{thm:01}. Recalling Sylvester's
rank inequality of the product of two matrices and its condition
of the equality, we are inspired to multiply (\ref{eq:wbc1}) left
by $\+M^T\+S^{-1}$ to get $n$ boundary conditions
\begin{equation}\label{eq:wbc2}
	\+M^T\+S^{-1}\+M(\+w_o(0)+\+f_o) = -b(\chi)\+M^T(\+w_e(0)+\+f_e),
\end{equation}
which satisfy the following well-posedness theorem: 
\begin{theorem}\label{thm:32}
  Suppose $\chi\in(0,1]$ and $r_2=0$ in \eqref{eq:G2}.  Then there
  exists a unique $\+G_e^T\+f_e$ such that the system \eqref{eq:KL_p2}
  with Maxwell-type boundary conditions \eqref{eq:wbc2} has a unique
  solution of $\+w(y)$.
\end{theorem}

\begin{proof}
	Denote by $\+H=\+M^T\+S^{-1}\+M$ and $\+B=[b(\chi)\+M^T,
	\+H]\in\bbR^{n\times(m+n)}$. By Theorem \ref{thm:01},
	we substitute $\+V_2^T\+w=\+0$ and $\+V_3^T\+w=\+T_+\+z$ into 
	(\ref{eq:wbc2}). Due to the block structure
	(\ref{eq:G2})(\ref{eq:V3}), we write $\+f_e=\+G_e
	\+G_e^T\+f_e+(\+I_m-\+G_e\+G_e^T)\+f_e$. Then the boundary condition
	(\ref{eq:wbc2}) becomes
	\begin{eqnarray}\label{eq:tem22}
		&&b(\chi)\+M^T\+G_e(\+X_e\+X_e^T\+G_e^T\+w_e(0)+\+G_e^T\+f_e) +
		\+B\+V_3 (\+T_+\+z(0)) \notag \\
		&=& -\+H\+f_o-b(\chi)\+M^T(\+I_m-\+G_e\+G_e^T)\+f_e.
	\end{eqnarray}
	We only need to claim that $\mathrm{rank}(\+B\+V_3\+T_+)=
	\mathrm{rank}(\+T_+)$ and $\mathrm{rank}([\+M^T\+G_e,
	\+B\+V_3\+T_+])=n$. 
	
	In fact, if the two claims are right, when $b(\chi)=0,$ we can 
	solve a unique $\+z(0)$ if $\+H\+f_o\in\mathrm{span}\{\+B\+V_3
	\+T_+\}$. When $b(\chi)>0$, we can uniquely solve 
	$\+G_e^T\+f_e$ and $\+z(0)$ for arbitrary given $\+f_o$ and
	$(\+I_m-\+G_e\+G_e^T)\+f_e$, since $\+Y_1^T\+w_e$ is determined
	(\ref{eq:V1T}) by $\+z$ and $\+G_e^T\+G_e=\+I_{p_1}$.
	
	We first show $\mathrm{rank}(\+B\+V_3\+T_+)=\mathrm{rank}(\+T_+).$
	Let $\tilde{m}=m-r_1-p_2$. By (\ref{eq:T}), we assume $\+T_+=\+L^{-T}
	\+R_+$ with $\+L^{-1}\+A_{33}\+L^{-T}\+R_+=\+R_+\+\Lambda_+$ and
	$\+R_+^T\+R_+=\+I_{n_+}$. Due to the block structure
	of $\+Q_{33}=\+L\+L^T$, we assume $\+L=\mathrm{diag}(\+L_e,
	\+L_o)$ with $\+L_e\in\bbR^{\tilde{m}\times \tilde{m}}$ and
	$\+L_o\in\bbR^{n_+\times n_+}$. Since $\+L^{-1}\+A_{33}\+L^{-T}$
	has a block structure, its positive and negative
	eigenvalues always appear in pair and have a special relation 
	(cf. proof of Lemma \ref{lem:41}). So 
	if we write $\+R_+ = \begin{bmatrix}
	\+R_e\\ \+R_o \end{bmatrix}$ where $\+R_e\in\bbR^{\tilde{m}\times
	n_+}$ and $\+R_o\in\bbR^{n_+\times n_+}$, we can assume
	$\+R_e^T\+R_e=\+R_o^T\+R_o=\frac{1}{2}\+I_{n_+}$.
	Since $\+S>0$ and $\mathrm{rank}(\+M)=n$, we have $\+H>0.$
	Now we have 
	\begin{eqnarray*}
		\+Z_3^T\+B\+V_3\+T_+  
		&=& b(\chi)\+Z_3^T\+M^T\+Y_3\+L_e^{-T}\+R_e + 
		\+Z_3^T\+H\+Z_3\+L_o^{-T}\+R_o \\
		&=& b(\chi)\+L_o\+R_o\+\Lambda_+ +
		\+Z_3^T\+H\+Z_3\+L_o^{-T}\+R_o.
	\end{eqnarray*}
	For any $\+x\in\bbR^{n_+}$, 
	since $\+\Lambda_+>0$ and $b(\chi)\geq 0$ when
	$\chi\in[0,1]$, we have
	\begin{eqnarray}\notag
		\+x^T\+R_o^T\+L_o^{-1}\+Z_3^T\+B\+V_3\+T_+\+x &=&
		\frac{1}{2}b(\chi)\+x^T\+\Lambda_+\+x
		+ \+x^T\+R_o^T\+L_o^{-1}\+Z_3^T\+H\+Z_3\+L_o^{-T}\+R_o\+x \\
		&\geq& \+x^T\+R_o^T\+L_o^{-1}\+Z_3^T\+H\+Z_3\+L_o^{-T}\+R_o\+x
		\geq 0. 
		\label{eq:tm33}
	\end{eqnarray}
	The equality holds if and only if $\+x=\+0$ since $\+Z_3$ is also
	of full column rank. Hence,
	\[ n_+=\mathrm{rank}(\+R_o^T\+L_o^{-1}\+Z_3^T\+B\+V_3\+T_+)\leq
	\mathrm{rank}(\+B\+V_3\+T_+)\leq n_+\quad\Rightarrow\quad
	\mathrm{rank}(\+B\+V_3\+T_+)=n_+.
	\]

	Then we show that $\mathrm{rank}([\+B\+V_3\+T_+,\+M^T\+G_e])=n.$
	Since $r_2=0$ and $\mathrm{rank}(\+M^T\+G_e)=p_1,$ it's enough to
	show that 
	\[\mathrm{span}\{\+B\+V_3\+T_+\}\cap\mathrm{span}\{\+M^T\+G_e\}
	=\{\+0\}.\] 
	If $\+x\in\mathrm{span}\{\+M^T\+G_e\}=\mathrm{span}\{\+
	Z_2\}$, then we have $\+Z_3^T\+x=\+0$ since $\+Z_3^T\+Z_2=\+0.$ 
	But from (\ref{eq:tm33}), if this $\+x$ also belongs to
	$\mathrm{span}\{\+B\+V_3\+T_+\}$, then $\+x=\+0.$ This completes
	the proof.
\hfill
\end{proof}

\begin{remark} \label{rem:33}
	The modification of \eqref{eq:obc} is not unique. 
	From the above proof, we know that for any 
	symmetric positive definite matrix $\+H\in\bbR^{n\times n}$,
	the boundary condition
	\[\+H(\+w_o(0)+\+f_o)=-b(\chi)\+M^T(\+w_e(0)+\+f_e)\]
	would satisfy a theorem similarly
	as Theorem \ref{thm:32}. 
	We leave the comparison of different modifications elsewhere.
	\renew{We note that the solvable Maxwell-type boundary condition
	with an even-odd parity form appears commonly in literature for many
	other problems \cite{Bern2010,Li2017,Sarna2018}}.

	Note that $\+M$ is invertible when $m=n$.
	So the new boundary condition \eqref{eq:wbc2} is equivalent to
	Grad's boundary condition \eqref{eq:obc} when $m=n$. This is also
	an advantage of the modification.
\end{remark}

\section{Applications}
\subsection{Kramers' Problem.}
\emph{Kramers' problem} \cite{Kramers1949} concerns the tangential 
velocity of a gas flow when passing by an infinite plate, with the 
only driven force from the tangential stress. For the Shakhov model,
when $D=3$, we can choose the index set $\mathbb{I}$ as
\begin{equation*}
	\mathbb{I}=\{\+e_1+2\+e_j+\alpha_2\+e_2,\ j=1,3,\ 0\leq\alpha_2
\leq M-3\}\cup\{\+e_1+\alpha_2\+e_2,\ 0\leq\alpha_2\leq M-1\}.
\end{equation*}
This choice gives $p_1=1,\ p_2=0,\ p=1$ and
$r_1=1,\ r_2=0,\ r=1$. By Theorem
\ref{thm:01}, we can solve $p+r=2$ variables by (\ref{eq:V1T})
and (\ref{eq:V2T}):
\begin{eqnarray}
\sigma_{12,K}(y) &=& 0,\\ \label{eq:uK}
u_{1,K}(y) &=& -\sqrt{2}\+w[\+e_1+2\+e_2](y).
\end{eqnarray}
In (\ref{eq:wbc2}), we suppose the wall is motionless. Let non 
zero entries of $\+f_e$ and $\+f_o$ be
$\+f_e[\+e_1]=u_{1,B}$ and  
$\+f_o[\+e_1+\+e_2]=\sigma_{12,B}$, where 
${\sigma}_{12,B}$ is given.
Now $\+G_e^T\+f_e=\+f_e[\+e_1]=u_{1,B}.$ 
By Theorem \ref{thm:32}, when $0<\chi\leq 1$,
we can solve constants $\+z\in\bbR^{n-1}$ and $c_0\in\bbR$ 
depending on $\sigma_{12,B}$ that
\begin{eqnarray*} 
	{u}_{1,K}(y) = \+c_1^T\exp\left(-\+\Lambda_+^{-1}
	y\right)\+z; \qquad
	{u}_{1,B} = c_0, 
\end{eqnarray*}
where $\+c_1$ and $\+\Lambda_+$ are determined by the proof process
of Theorem \ref{thm:01}.

\begin{remark}
Under the above settings, we can formally establish the relationship
between the moment system and kinetic equations for
Kramers' problem \cite{Siewert2001}. The details can refer to
\cite{Lijun2017}.
\end{remark}

Now we define the viscous slip coefficient $\eta$ as
\begin{equation}
	\eta= -\mu {u}_{1,B}/{{\sigma}_{12,B}},
\end{equation}
where $\mu$ is the viscosity coefficient determined by the collision
term. For the Shakhov model with different $\Pr$, the value of $\mu$
should be same (cf. \cite{Loyalka1990,Hattori2019}). We let 
$\mu=\sqrt{2}/2$ to follow Siewert's results \cite{Siewert2001}.
Below we use the Shakhov model to refer to the case $\Pr=2/3$ 
especially. 
We also consider the normalized 
velocity profile in the Knudsen layer, i.e., the velocity defect
\begin{equation}
	u_d(y) = \mu {u}_{1,K}(y)/{\sigma}_{12,B}.
\end{equation}

\begin{figure}[!htb]
\pgfplotsset{width=0.48\textwidth}
\centering
\begin{tikzpicture} 
\begin{axis}[
    xlabel=$M$, 
    ylabel=$\eta$, 
    tick align=outside, 
    legend style={at={(0.73,0.42)},anchor=north} 
    ]
\addplot[smooth,mark=*,blue] table {pic/BGK_K_001.dat};
\addlegendentry{BGK}
\addplot[smooth,mark=square,red] table {pic/Sh_K_001.dat};
\addlegendentry{Shakhov}
\addplot[smooth,black,dashed,ultra thick]
coordinates{(0,1.01619) (150,1.01619) };
\addlegendentry{ref: BGK}
\addplot[smooth,black,dotted,ultra thick]
coordinates{(0,1.01837) (150,1.01837) };
\addlegendentry{ref: Shakhov}
\end{axis}
\end{tikzpicture}
\hskip 3pt %
\begin{tikzpicture} 
\begin{axis}[
    xlabel=$\mu y$, 
    ylabel=$u_d$, 
    tick align=outside, 
    legend style={at={(0.7,0.9)},anchor=north} 
    ]
\addplot[smooth,mark=*,blue] table {pic/BGK_K_002.dat};
\addlegendentry{BGK}
\addplot[smooth,mark=square,red] table {pic/Sh_K_002.dat};
\addlegendentry{Shakhov}
\end{axis}
\end{tikzpicture}
\caption{ \emph{Left}: Slip coefficients of different even
   	$M$ for various models when $\chi=1$. \emph{Right}: 
	The normalized velocity profile when $\chi=1$ and
	$M=80$.
}
\label{fig:K1}
\end{figure}
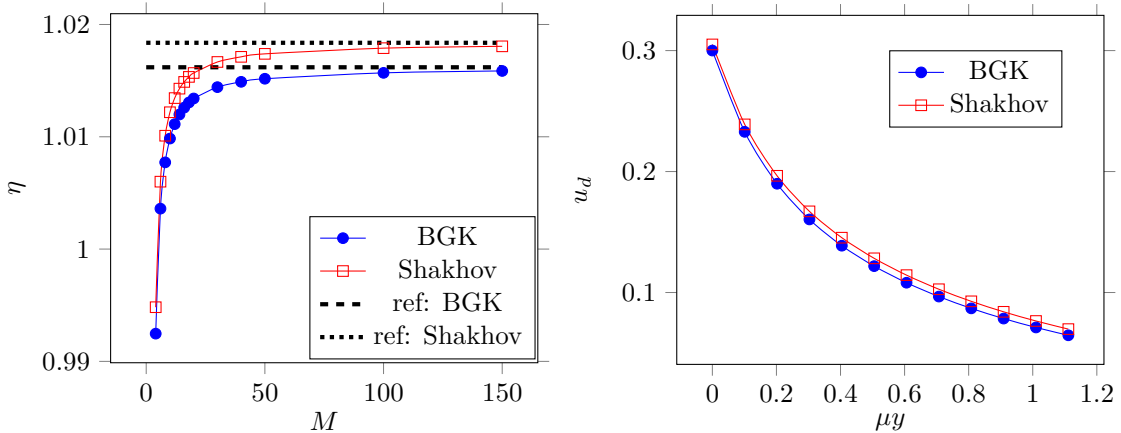

Fig.\ref{fig:K1} exhibits the viscous slip coefficient and
the normalized velocity profile calculated
by our method when $\chi=1$ and $M$ is an even number.
The reference values are
$\eta_{BGK}=1.01619$ and $\eta_{Sh}=1.01837$ (cf.
\cite{Siewert2001,Hattori2019}).
From Fig.\ref{fig:K1}, we see a converging
trend when $M$ increases, and the relative error 
compared with the reference value would be
lower than $1\%$ when $M>10.$ We also find that 
the Shakhov model has a Knudsen layer thicker than the BGK model. 

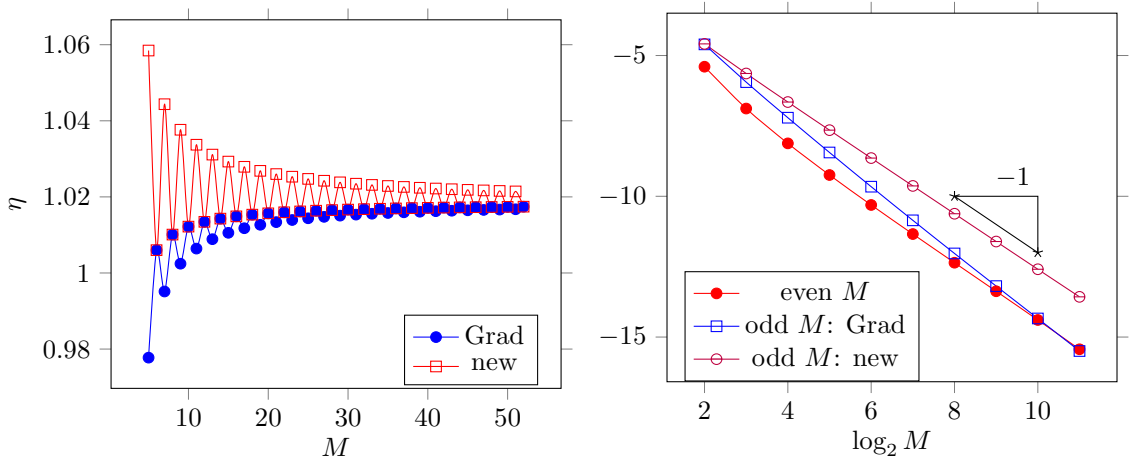
\begin{figure}[!htb]
\pgfplotsset{width=0.5\textwidth}
\centering
\begin{tikzpicture} 
\begin{axis}[
    xlabel=$M$, 
    ylabel=$\eta$, 
    tick align=outside, 
    legend style={at={(0.8,0.2)},anchor=north} 
    ]
\addplot[smooth,mark=*,blue] table {pic/K1_old.dat};
\addlegendentry{Grad}
\addplot[smooth,mark=square,red] table {pic/K1_new.dat};
\addlegendentry{new}
\end{axis}
\end{tikzpicture}
\hskip 3pt %
\begin{tikzpicture} 
\begin{axis}[
    xlabel=$\log_2 M$, 
    tick align=outside, 
    legend style={at={(0.3,0.3)},anchor=north} 
    ]
\addplot[smooth,mark=*,red] table {pic/K2_e.dat};
\addlegendentry{even $M$}
\addplot[smooth,mark=square,blue] table {pic/K2_o.dat};
\addlegendentry{odd $M$: Grad}
\addplot[smooth,mark=halfcircle,purple] table {pic/K2_no.dat};
\addlegendentry{odd $M$: new}
\addplot plot coordinates {(8,-10) (10,-12)}
coordinate [pos=0] (A)
coordinate [pos=1] (B)
;
\draw (A) -| (B)
node [pos=0.35,anchor=south]
{\pgfmathprintnumber{-1}};
\end{axis}
\end{tikzpicture}

\caption{\emph{Left}: Slip coefficients $\eta$ for the BGK model
when $M$ ranges from 5 to 52. \emph{Right}: Log-log error 
	diagram of the slip coefficients for the BGK model. 
	y axis: $a_k$ or $b_k$. ($\chi=1$).}
\label{fig:K2}
\end{figure}

In Kramers' problem, we have $m=n$ when $M$ is an even number. So as 
shown in Remark \ref{rem:33}, the new boundary condition (\ref{eq:wbc2})
is equivalent to Grad's boundary condition (\ref{eq:obc}) when $M$
is even. Fig.\ref{fig:K2} compares the slip coefficients obtained by
different boundary conditions for the BGK model. 
Denoting by $\eta_M$ the slip
coefficient for a given $M$, we calculate
$a_{k-1}=\log_2|\eta_{BGK}-\eta_{2^k}|$ and $b_{k-1}=\log_2
|\eta_{BGK}-\eta_{2^k+1}|$ for $k=2,3,..,11.$ From Fig.\ref{fig:K2},
we see that when $M$ is an odd number, the new boundary condition
(\ref{eq:wbc2}) brings a totally different picture as in the even
case. Regardless of the choices,
we always observe the first-order convergence of $\eta_M$ when $M$
increases. The convergence rate may be affected by the discontinuity
of the velocity distribution function at the wall.

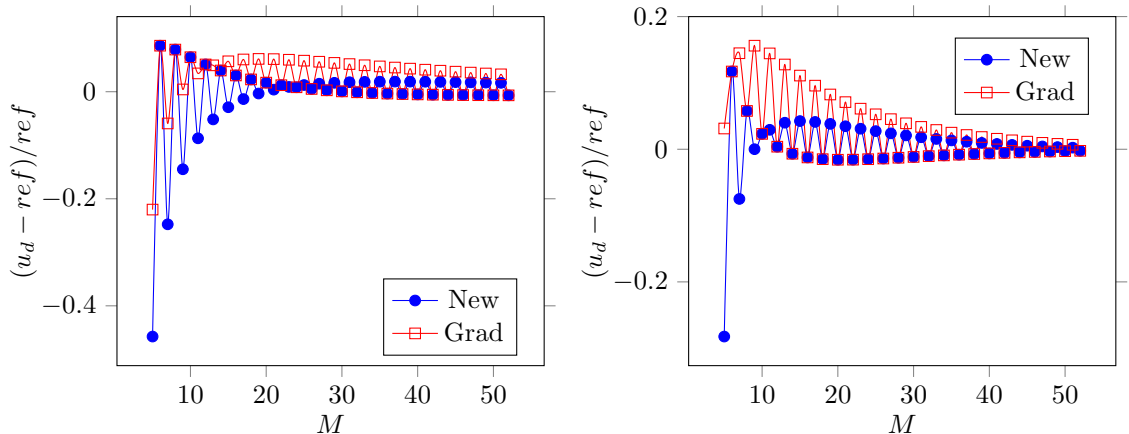
\begin{figure}[htb!]
\pgfplotsset{width=0.48\textwidth}
\centering
\begin{tikzpicture} 
\begin{axis}[
    xlabel=$M$, 
    ylabel=$(u_d-ref)/ref$, 
    tick align=outside, 
    legend style={at={(0.78,0.25)},anchor=north} 
    ]
\addplot[smooth,mark=*,blue] table {pic/B_401.dat};
\addlegendentry{New}
\addplot[smooth,mark=square,red] table {pic/B_402.dat};
\addlegendentry{Grad}
\end{axis}
\end{tikzpicture}
\hskip 3pt %
\begin{tikzpicture} 
\begin{axis}[
    xlabel=$M$, 
    ylabel=$(u_d-ref)/ref$, 
    tick align=outside, 
    legend style={at={(0.78,0.94)},anchor=north} 
    ]
\addplot[smooth,mark=*,blue] table {pic/B_403.dat};
\addlegendentry{New}
\addplot[smooth,mark=square,red] table {pic/B_404.dat};
\addlegendentry{Grad}
\end{axis}
\end{tikzpicture}
\caption{ Relative errors of the velocity defect for the BGK model
	when $\chi=0.1$ and $M$ ranges from 5 to 52.
	\emph{Left}: At $\mu y=0.5.$
   	\emph{Right}: At $\mu y=1.0.$
}
\label{fig:K4}
\end{figure}

From Fig.\ref{fig:K2}, we may conclude that the new boundary condition
gives a relatively large error in the slip coefficient, compared to
the Grad's one when $M$ is odd. But we find that it also gives a better 
approximation to the velocity defect $u_d(y)$ when $y$ is away from zero.
This is shown in Fig.\ref{fig:K4}, where the relative errors of $u_d(y)$
at $y=0.5\mu^{-1}$ and $y=\mu^{-1}$ are calculated for the BGK model 
with different $M$ when $\chi=0.1$. The reference values are 
from \cite{Siewert2001}.

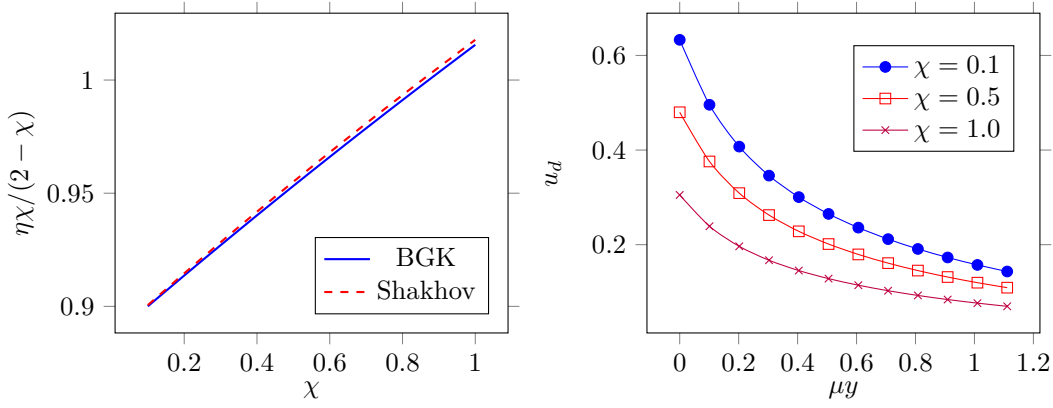
\begin{figure}[htb!]
\pgfplotsset{width=0.45\textwidth}
\centering
\begin{tikzpicture} 
\begin{axis}[
    xlabel=$\chi$, 
    ylabel=$\eta\chi/(2-\chi)$, 
    tick align=outside, 
    legend style={at={(0.73,0.3)},anchor=north} 
    ]
\addplot[smooth,blue,thick] table {pic/BGK_ex1.dat};
\addlegendentry{BGK}
\addplot[smooth,dashed,red,thick] table {pic/Sh_ex1.dat};
\addlegendentry{Shakhov}
\end{axis}
\end{tikzpicture}
\hskip 3pt %
\begin{tikzpicture} 
\begin{axis}[
    xlabel=$\mu y$, 
    ylabel=$u_d$, 
    tick align=outside, 
    legend style={at={(0.73,0.9)},anchor=north} 
    ]
\addplot[smooth,mark=*,blue] table {pic/Sh_K_031.dat};
\addlegendentry{$\chi=0.1$}
\addplot[smooth,mark=square,red] table {pic/Sh_K_032.dat};
\addlegendentry{$\chi=0.5$}
\addplot[smooth,mark=x,purple] table {pic/Sh_K_033.dat};
\addlegendentry{$\chi=1.0$}
\end{axis}
\end{tikzpicture}

\caption{\emph{Left}: $\eta\chi/(2-\chi)$.
	\emph{Right}: Velocity profile of the Shakhov model when $\chi$
		varies. ($M=80$).
}
\label{fig:K3}
\end{figure}

Fig.\ref{fig:K3} shows the velocity profile and slip coefficient
when the accommodation coefficient $\chi$ varies and
$M=80.$ We see that $\eta$ goes larger 
and the Knudsen layer becomes thicker when $\chi$ goes smaller.
This observation coincides with the classical qualitative results.

\subsection{Thermal Slip Problem.}
We consider the tangential velocity of the flow caused by a
temperature gradient in a direction parallel to the wall, which
is called \cite{Williams2001} the \emph{thermal slip problem}.
For the Shakhov model, when $D=3$, we choose the same $\mathbb{I}$
as in Kramers' problem:
\begin{equation*}
	\mathbb{I}=\{\+e_1+2\+e_j+\alpha_2\+e_2,\ j=1,3,\ 0\leq\alpha_2
\leq M-3\}\cup\{\+e_1+\alpha_2\+e_2,\ 0\leq\alpha_2\leq M-1\}.
\end{equation*}
The difference lies in the boundary condition (\ref{eq:wbc2}). We 
let $\+f_o$ be zero and non zero entries of $\+f_e$ be 
$\+f_e[\+e_1]=u_{1,B},\ \+f_e[3\+e_1]=\sqrt{3}q_{1,B},$ and 
$\+f_e[\+e_1+2\+e_i]=q_{1,B},\ i\neq 1,$ where $q_{1,B}$ is
a given constant. By Theorem \ref{thm:32}, when $0<\chi\leq 1$,
we can solve constants $\+z\in\bbR^{n-1}$ and $c_0\in\bbR$ 
depending on $q_{1,B}$ that
\begin{eqnarray*}
	{u}_{1,K}(y) = \+c_1^T\exp\left(-\+\Lambda_+^{-1}
	y\right)\+z, \qquad
	{u}_{1,B} = c_0, 
\end{eqnarray*}
where $\+c_1$ and $\+\Lambda_+$ are determined by the proof process
of Theorem \ref{thm:01}. 

Then we can define the thermal slip coefficient as
\begin{equation}
	\eta_t = -\frac{1}{2}\lambda u_{1,B}/q_{1,B},
\end{equation}
where $\lambda$ is the thermal conductivity coefficient determined
by the collision term. To agree with \cite{Siewert2002}'s results
about kinetic equations, we let $\lambda={\Pr}^{-1}\sqrt{2}/2$ and
represent the value of $\Pr\eta_t.$

\begin{table}[!htb] 
\centering 
\caption{The thermal slip coefficient $\Pr\eta_t$ compared with
	\cite{Siewert2002}}\label{tab:02}
\begin{tabular}{ccccccc} 
\toprule 
	$\chi$ & BGK-\cite{Siewert2002} & S-\cite{Siewert2002} & BGK$-M=12$ &
	S$-M=12$ & BGK$-M=84$ & S$-M=84$ \\ 
\midrule 
0.1 & 0.264178 & 0.266064 & 0.263578 & 0.265470 & 0.264101 & 0.265989\\
0.2 & 0.278151 & 0.281655 & 0.277030 & 0.280570 & 0.278009 & 0.281521\\
0.3 & 0.291924 & 0.296794& 0.290360 & 0.295311 & 0.291728 &  0.296615\\
0.4 & 0.305502 & 0.311501& 0.303568 & 0.309703 & 0.305263 &  0.311287\\
0.5 & 0.318891 & 0.325791& 0.316657 & 0.323758 & 0.318619 &  0.325554\\
0.6 & 0.332095 & 0.339683& 0.329630 & 0.337485 & 0.331799 &  0.339431\\
0.7 & 0.345120 & 0.353193& 0.342487 & 0.350894 & 0.344808 &  0.352935\\
0.8 & 0.357969 & 0.366335& 0.355231 & 0.363994 & 0.357650 &  0.366077\\
0.9 & 0.370648 & 0.379125& 0.367863 & 0.376794 & 0.370328 &  0.378873\\
1.0 & 0.383161 & 0.391575& 0.380287 & 0.389303 & 0.382847 &  0.391335\\
\bottomrule 
\end{tabular}
\end{table}

\begin{figure}[htb!]
	\centering
\begin{tikzpicture} 
\pgfplotsset{width=0.48\textwidth}
\begin{axis}[
    xlabel=$M$, 
    ylabel=$\eta_t$, 
    tick align=outside, 
    legend style={at={(0.8,0.2)},anchor=north} 
    ]
\addplot[smooth,mark=*,blue] table {pic/T5_B_old.dat};
\addlegendentry{Grad}
\addplot[smooth,mark=square,red] table {pic/T5_B_new.dat};
\addlegendentry{new}
\end{axis}
\end{tikzpicture}
\hskip 3pt %
\begin{tikzpicture} 
\pgfplotsset{width=0.48\textwidth}
\begin{axis}[
    xlabel=$M$, 
    ylabel=$\eta_t$, 
    tick align=outside, 
    legend style={at={(0.8,0.2)},anchor=north} 
    ]
\addplot[smooth,mark=*,blue] table {pic/T5_S_old.dat};
\addlegendentry{Grad}
\addplot[smooth,mark=square,red] table {pic/T5_S_new.dat};
\addlegendentry{new}
\end{axis}
\end{tikzpicture}

\caption{ The thermal slip coefficient when $M$ range from 5 to 53
	and $\chi=1.$ \emph{Left}: The BGK model.
   \emph{Right}: The Shakhov model with $\Pr=2/3.$
}
	\label{fig:T10}
\end{figure}
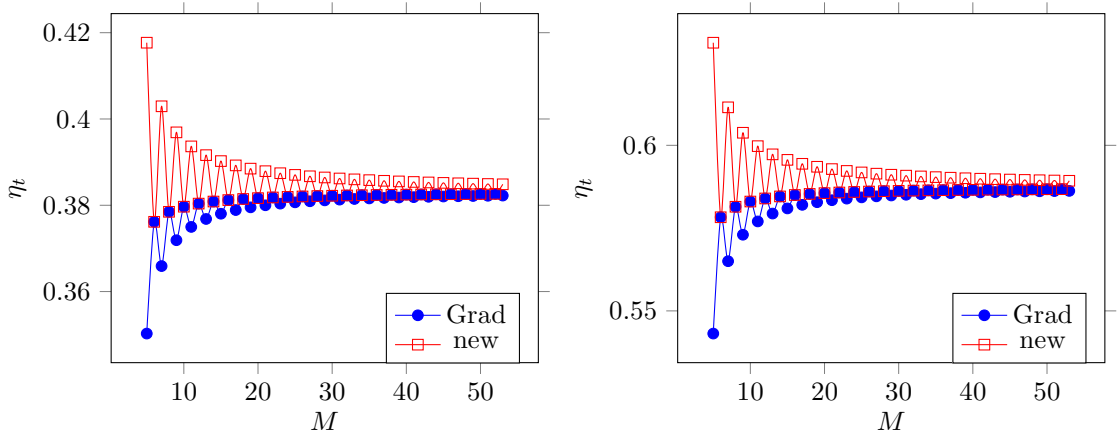

Table.\ref{tab:02} compares our results with \cite{Siewert2002}, where
the letter S means the Shakhov model with $\Pr=2/3.$ We see that our 
method already gives high accuracy results when $M=12$. Fig.\ref{fig:T10} shows
the converging trend when $M$ goes to infinity with $\chi=1$. We
can see the influence of different boundary conditions again from
Fig.\ref{fig:T10}.

\subsection{Temperature Jump Problem.}
The \emph{temperature jump problem} \cite{Welander1954} 
concerns the gas temperature near the wall when the flow passes
over an infinite plate, with the only driven force from a normal 
temperature gradient. 
For the Shakhov model, when $D=3$, the index set could be
\begin{equation}\label{eq:I22}
	\mathbb{I}=\{2\+e_j+\alpha_2\+e_2,\ j=1,3,\ 0\leq\alpha_2\leq M-2\}
\cup\{\alpha_2\+e_2,\ 0\leq\alpha_2\leq M\},
\end{equation}
which gives $p_1=2,\ p_2=1$ and $r_1=1,\ r_2=0$.
Three variables are solved by (\ref{eq:V2T}): 
\begin{eqnarray}
	{u}_{2,K}(y) &=& 0,\\
	{\rho}_K(y) + \sqrt{2}\+w[2\+e_2](y) &=& 0,\\
	5{q}_{2,K}(y) &:=& \sqrt{3}\+w[3\+e_2]
	+\+w[\+e_2+2\+e_1]+\+w[\+e_2+2\+e_3]= 0.
\end{eqnarray}
One variable $\sqrt{2} {\rho}_K(y) - \sum_{i}\+w[2\+e_i](y)$ is 
solved by (\ref{eq:V1T}), where $\sqrt{2}\sum_i \+w[2\+e_i] = 
3{\theta}_{K}$ according to (\ref{eq:LHME_h}).
In (\ref{eq:wbc2}), we suppose the wall temperature do not
change. We let $\+f_e=\+G_e\+G_e^T\+f_e$ in (\ref{eq:tem22})
with $\+f_e[2\+e_i]=\theta_B/\sqrt{2}$. Then we let
non zero entries of $\+f_o$ be $\+f_o[\+e_2+2\+e_i]=q_{2,B},
\ i\neq 2,$ and $\+f_o[3\+e_2]=\sqrt{3}q_{2,B}$, where
$q_{2,B}$ is given. By Theorem \ref{thm:32}, when $0<\chi\leq 1$, we can 
solve $\+z\in\bbR^{n-2}$ and $c_1\in\bbR$ depending on $q_{2,B}$
that 
\begin{eqnarray*}
	{\theta}_{K}(y) = \+c_1^T\exp\left(-\+\Lambda_+^{-1}
	y\right)\+z, \qquad 
	{\theta}_{B} = c_1,
\end{eqnarray*}
where $\+c_1$ and $\+\Lambda_+$ 
can be determined by the proof process of Theorem \ref{thm:01}.

Analogously, we define the temperature jump coefficient $\zeta$ as
\begin{equation}
	\zeta=-\frac{1}{\sqrt{2}}\lambda{\theta}_{B}/{{q}_{2,B}},
\end{equation}
where $\lambda$ is the thermal conductivity coefficient determined
by the collision term. Due to the conversion rule \cite{Hattori2019}, 
we let $\lambda={\Pr}^{-1}\sqrt{2}/2$. Under the choice
of (\ref{eq:I22}), Knudsen layer equations (\ref{eq:KL_p2}) in the
Shakhov case are the same as in the BGK case. Thus, 
we immediately have 
\begin{equation}
	\zeta(\Pr) = {\Pr}^{-1}\zeta(1),
\end{equation}
where $\zeta$ is viewed as a function of $\Pr$ in the Shakhov case.
This result coincides with \cite{Hattori2019} and leads us to only
consider the jump coefficient in the BGK case.
The normalized temperature profile in the Knudsen layer, or
called the temperature defect, is defined as 
\begin{equation}
	\theta_d(y) = \frac{\lambda}{\sqrt{2}}{\theta}_{K}(y)/{q}_{2,B}.
\end{equation}

Now $m=n$ when $M$ is an odd number, where the boundary condition
(\ref{eq:wbc2}) is equivalent to Grad's (\ref{eq:obc}). So we first
consider the odd case. Table.\ref{tab:01} compares our results of the
jump coefficient with the kinetic results \cite{Siewert2000}. We see 
that the jump
coefficient goes larger when $\chi$ goes smaller. Our solutions seem
to agree with reference solutions well with not too many moments. 
In fact, when $M=13$, the relative error is less than $1\%$ in most
cases.

\begin{table}[!htb] 
\centering 
\caption{The temperature jump coefficient compared with Barichello 
and Siewert's results \cite{Siewert2000}}\label{tab:01}
\begin{tabular}{cccccccc} 
\toprule 
	$\chi$ & \cite{Siewert2000} & $M=3$ & $M=5$ &
	$M=7$ & $M=9$ & $M=11$& $M=13$\\ 
\midrule 
0.1 & 21.4501 & 21.0856 & 21.3565 & 21.3957 & 21.4119 & 21.4208 &
	21.4263\\
0.3 & 6.63051 & 6.31159& 6.55416 & 6.58698 & 6.60028 & 6.60742 &
	6.61185\\
0.5 & 3.62913 & 3.35382& 3.56804 & 3.59507 & 3.60574 & 3.61139 &
	3.61487\\
0.6 & 2.86762 & 2.61342& 2.81345 & 2.83779 & 2.84726 & 2.85224 &
	2.85529\\
0.7 & 2.31753 & 2.08401& 2.26984 & 2.29162 & 2.29997 & 2.30432 &
	2.30698\\
0.9 & 1.57036 & 1.37681& 1.53420 & 1.55126 & 1.55758 & 1.56081 &
	1.56276\\
1.0 & 1.30272 & 1.12868& 1.27183 & 1.28673 & 1.29213 & 1.29488 &
	1.29652\\
\bottomrule 
\end{tabular}
\end{table}

Fig.\ref{fig:T1} presents the results of different boundary
conditions, where the reference solution is from
\cite{Siewert2000}. From the left picture,
we again see that when $M$ is an even number, the jump coefficient 
corresponding to the new boundary condition displays a decreasing
converging trend. From the right picture, we see that despite 
a relatively larger gap at $y=0$, the temperature defect
of the new boundary condition approaches faster to 
the reference solution than in the Grad's case. 

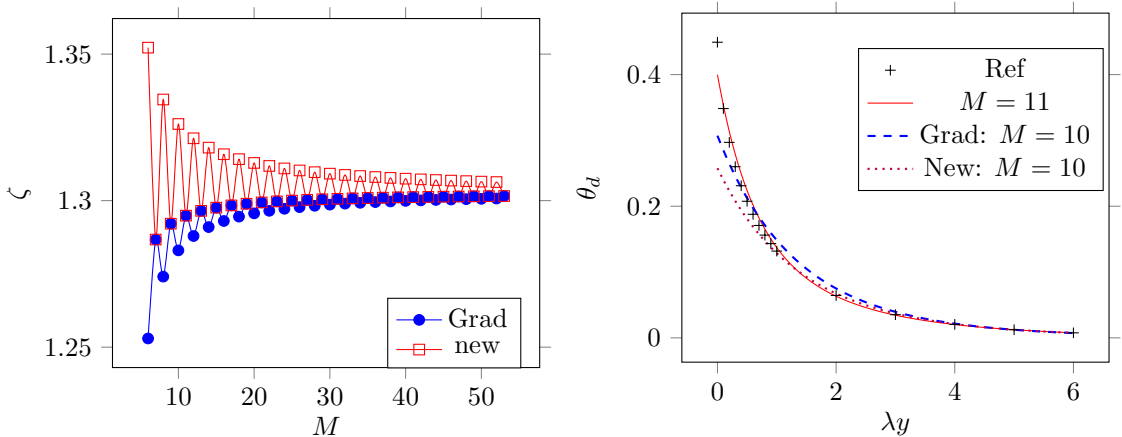
\begin{figure}[htb!]
	\centering
\begin{tikzpicture} 
\pgfplotsset{width=0.48\textwidth}
\begin{axis}[
    xlabel=$M$, 
    ylabel=$\zeta$, 
    tick align=outside, 
    legend style={at={(0.8,0.2)},anchor=north} 
    ]
\addplot[smooth,mark=*,blue] table {pic/T1_old.dat};
\addlegendentry{Grad}
\addplot[smooth,mark=square,red] table {pic/T1_new.dat};
\addlegendentry{new}
\end{axis}
\end{tikzpicture}
\hskip 3pt %
\begin{tikzpicture} 
\pgfplotsset{width=0.48\textwidth}
\begin{axis}[
    xlabel=$\lambda y$, 
    ylabel=$\theta_d$, 
    tick align=outside, 
    legend style={at={(0.7,0.9)},anchor=north} 
    ]
\addplot[only marks,mark=+,black] coordinates {
(0,0.449201) (0.1,0.348416) (0.2,0.297036) (0.3,0.260056)
(0.4,0.231136) (0.5,0.207516) (0.6,0.187706) (0.7,0.170756)
(0.8,0.156066) (0.9,0.143196) (1,0.131806) (2,0.064406)
(3,0.035206) (4,0.020426) (5,0.012306) (6,0.007626)
};
\addlegendentry{Ref}
\addplot[smooth,red] table {pic/041.dat};
\addlegendentry{$M=11$}
\addplot[smooth,blue,dashed,thick] table {pic/042.dat};
\addlegendentry{Grad: $M=10$}
\addplot[smooth,purple,dotted,thick] table {pic/043.dat};
\addlegendentry{New: $M=10$}
\end{axis}
\end{tikzpicture}

\caption{\emph{Left}:The temperature jump coefficient $\zeta$
	when $M$ ranges from 6 to 53.
   \emph{Right}: The temperature defect $\theta_d(y)$ for different
   boundary conditions. ($\chi=1$).
}
	\label{fig:T1}
\end{figure}

\subsection{Remarks.}
We remark three points which are observed from the numerical
results:
\begin{enumerate}
\item We may only need a moderate number of moments, e.g., $10<M<20$,
  to depict the Knudsen layer quite accurately. This empirical choice
  of $M$ may benefit the practical usage.
\item \renew{If $m>n$, the different boundary conditions would affect
	the results. When $y$ is near zero, the new boundary conditions
	may provide relatively greater errors about the slip/jump 
	coefficients or the profile of Knudsen layer solutions. 
	When $y$ goes far away from zero, the new boundary condition
	provides a better approximation, as shown in Fig. \ref{fig:K4}
	and Fig. \ref{fig:T1} with $y>1$.}
\item There is always a converging trend when $M$ goes to infinity.
  A possible analysis of the convergence and accuracy is beyond the
  scope of this paper.
\end{enumerate}

\section{Conclusions}

We imposed a type of Maxwell boundary conditions on the linear half-space
moment systems. The proposed boundary condition is derived from the
classical diffuse-specular model and proved to satisfy some
well-posedness criteria of half-space boundary value problems. The
procedure has been applied to flow problems with the Shakhov collision
term. In this way, the model with new boundary condition was validated
that it can capture viscous and thermal Knudsen layers well with only
a few moments. It is straightforward to apply our model to other
Knudsen layer problems with various collision terms.

\section*{Acknowledgment}
\par This work is financially supported by the National Key R\&D
Program of China, Project Number 2020YFA0712000. We thank Dr.Jun Li
and Dr.Yizhou Zhou for their enthusiastic discussions with us.


\bibliographystyle{plain}
\bibliography{../../ATC}
\appendix

\section{Proofs of Lemmas}\label{app:D}
\subsection*{Proof of Lemma \ref{lem:01}}
\begin{proof}
	Since $\+Q\+G=\+0$ and $\mathrm{Null}(\+A)\cap \mathrm{Null}(\+Q)=\{\+0\}$,
	we have $\mathrm{rank}(\+A\+G)=\mathrm{rank}(\+G)=p$. Thus
	$\+V_2$ can be constructed by a Gram-Schmidt orthogonalization
	of $\+A\+G$. By definition,
	\[ (\+A\+G)^T\+V_1 = (\+A\+G)^T\+G\+X = \+0 \quad
	\Rightarrow \quad \+V_2^T\+V_1=\+0.\]
	Hence $[\+V_1,\+V_2]\in\bbR^{(m+n)\times (p+r)}$ is column
	orthogonal and there must be $m+n\geq p+r.$ Then $\+V_3\in
	\bbR^{(m+n)\times(m+n-p-r)}$ can be chosen as an orthogonal
	complement of $[\+V_1,\+V_2].$
	
	Now we show that $\+U=[\+U_1,\+U_2,\+V_3]$ is invertible,
	where $\+U_1=\+G,\ \+U_2=\+A\+G\+X.$ Since $\+V_3^T\+V_2=\+0,$ 
	we have $\+V_3^T\+U_2=\+0$. 
	To make $\+U$ invertible, it's enough to show that 
\begin{equation}\label{eq:v3g}
	\mathrm{span}\{\+V_3\}\cap\mathrm{span}\{\+G\}=\{\+0\}.
\end{equation}
In fact, if $\+G\+c_1=\+V_3\+c_2$ for some $\+c_1\in\bbR^p$ and $
\+c_2\in\bbR^{(m+n-p-r)}$, then 
	$(\+G\+c_1)^T\+A\+G=\+0$ since $\+V_3^T\+V_2=\+0$.
So $\+c_1\in\mathrm{span}\{\+X\}$ and  
$\+V_3\+c_2\in\mathrm{span}\{\+G\+X\}=\mathrm{span}\{\+V_1\}$. 
Thus, from $\+V_3^T\+V_1=\+0$ we have $\+c_2=\+0$, which 
implies (\ref{eq:v3g}).

Finally, we have $\mathrm{rank}(\+U_2^T\+A\+V_1) =
	\mathrm{rank}(\+U_2^T\+U_2) = r$. Meanwhile, $\+V_3^T
\+A\+V_1=\+V_3^T\+U_2=\+0$. Since $\mathrm{span}\{\+V_3\}
\cap\mathrm{span}\{\+G\}=\{\+0\}$ from (\ref{eq:v3g}),
we have $\+V_3^T\+Q\+V_3>0$. 
\end{proof}

\subsection*{Proof of Lemma \ref{lem:010}}
\begin{proof}
	Suppose  $\+Y_3^T\+M\+Z_3\+x = \+0$ for some 
	$\+x\in\bbR^{n-r_2-p_1}$. Since $[\+Y_1,\+Y_2,\+Y_3]$
	is orthogonal, there exists $\+x_1$ and $\+x_2$ such that
\begin{equation}\label{eq:tem2}
\+M\+Z_3\+x = \+Y_1\+x_1+\+Y_2\+x_2,\quad 
\+x_1\in\bbR^{r_1},\ \+x_2\in \bbR^{p_2}.
\end{equation}
Since $\+Z_3^T\+M^T\+Y_1=\+0$ and $\+Y_2^T\+Y_1=\+0$, we 
have $\+x_1=\+0.$
Note that the result (\ref{eq:v3g}) implies
$\mathrm{span}\{\+Z_3\}\cap\mathrm{span}\{\+G_o\}=\{\+0\}$
and $\mathrm{span}\{\+Y_2\}=\mathrm{span}\{\+M\+G_o\}$.
Then since $\+M$ and $\+Z_3$ are of full column rank,
there must be $\+x=\+0.$
So $\+Y_3^T\+M\+Z_3$ is of full column rank.
\end{proof}

\subsection*{Proof of Lemma \ref{lem:31}}
\begin{proof}
	According to the description before Lemma \ref{lem:31}, 
	we only need to show that $\mathrm{rank}(\+M)=n.$
	If not, there must be not all zero coefficients $r_{\+\beta}$,
	$\+\beta\in\mathbb{I}_o$, s.t.
	\begin{equation*}
		\ang{\xi_2\omega\sum_{\+\beta\in\mathbb{I}_o}
	r_{\+\beta}\phi_{\+\beta},\ \phi_{\+\alpha}}
	=0,\quad \forall \+\alpha\in\mathbb{I}_e.\end{equation*}
	However, for any $\+\gamma\in\mathbb{I}_e$ with $\gamma_2=0$,
	$\+M[\+\gamma,\+\beta]\neq 0$ if and only if
	$\+\beta=\+\gamma+\+e_2.$ So $r_{\+\gamma+\+e_2}=0$
	from the above relations. By induction we have
	$r_{\+\gamma+\beta_2\+e_2}=0,\ 1\leq\beta_2\leq M-|\+\gamma|,\
	\beta_2$ odd. Because of ({\bf C1}), any $\+\beta\in\mathbb{I}_o$
	can be represented in the above form $\+\gamma+\beta_2\+e_2$ and
	consequently, $r_{\+\beta}=0,\ \forall\+\beta\in\mathbb{I}_o$. 
	This contradiction
	shows that $\+M$ must have a full column rank of $n$.

	It's classical for the linearized Boltzmann 
	operator (cf. \cite[Chapter III]{Cercignani1969}) that $\+Q\geq 0.$
	It's also shown that
	$ \ang{\mathcal{L}[\omega\phi_{\+\beta}]
	\phi_{\+\alpha}} = 0,\ \+\alpha\in\mathbb{I}_e,
	\+\beta\in\mathbb{I}_o.$
	So $\+Q$ has a block diagonal structure. 
	
	Note that $\mathrm{Null}(\+Q)$ always has
	a constant dimension because
	of properties of the linearized
	Boltzmann operator. So direct calculations show 
	$\mathrm{Null}(\+A)\cap \mathrm{Null}(\+Q)=\{\+0\}.$ 
	When $\+J$ is replaced by $\+J_{Sh}$ in (\ref{eq:Q_Sh}), 
	direct calculation also shows the lemma right.
\end{proof}

\subsection*{Proof of Lemma \ref{lem:03}}
\begin{proof}
	By definition, $\+S$ is symmetric. For any $\+z\in \mathbb{R}^{m}$, 
	denote by
\begin{eqnarray*}
	f(\+\xi) = \sum_{\+\alpha\in\mathbb{I}_e} \+z[\+\alpha]
	\phi_{\+\alpha}(\+\xi),
\end{eqnarray*}
then we have $\displaystyle
	\+z^T \boldsymbol{S} \+z = \frac{\sqrt{2\pi}}{2}
	\ang{|\xi_2|\omega f^2}\geq 0.$ If $\+z^T \boldsymbol{S} \+z =0$,
	there must be $f=0.$ Due to the orthogonality, $f=0$ means
	$\+z=\+0.$ Hence, $\+S$ is symmetric positive definite.
\end{proof}

\section{Derivation of the Knudsen layer equations (\ref{eq:KL_pre})}
\label{app:F}
\renew{
The system (\ref{eq:KL_pre}) can be derived in several ways.
We put one method to derive it in the Appendix for completeness.
With notations in Definition 3.2, we consider the linearized 
Boltzmann equation
\[\pd{f}{t}+\sum_{d=1}^D\xi_d\pd{f}{x_d} = \mathcal{L}[f].\]
Choose an integer $M\geq 3$ and test the above equation with 
the orthonormal Hermite polynomials $\phi_{\+\alpha},\
|\+\alpha|\leq M$, then we have an unclosed moment system
\[\pd{\ang{f\phi_{\+\alpha}}}{t}+\sum_{d=1}^D
\pd{\ang{\xi_df\phi_{\+\alpha}}}{x_d} = 
\ang{\mathcal{L}[f]\phi_{\+\alpha}}.\]
Assume $\+f\in\bbR^N$ with $N=\#\{\+\alpha\in\bbN^D,|\+\alpha|\leq M\}$
and $\+f[\+\alpha]=\ang{f\phi_{\+\alpha}}$. Grad's ansatz
assumes $f=\omega\sum_{|\+\alpha|\leq M}\+f[\+\alpha]\phi_{\+\alpha}$,
which would close the above moment system, as
\[ \pd{\+f}{t} + \sum_{d=1}^D\+A_d\pd{\+f}{x_d} = -\+J\+f,\]
where $\+A_d$ is analogous to $\+A_2$ in (\ref{eq:LHME_A})
and $\+J$ is defined in (\ref{eq:Q}). After some careful definition,
macroscopic variables can be related to $\+f$ by similar formulas
like (\ref{eq:LHME_h}).}

Putting aside the detailed boundary condition, we make some 
Knudsen layer assumptions. Due to the rotation invariance, we
consider the plane-boundary lying at $\{\+x\in\bbR^D:x_2=0\}$,
introducing $\+x^w=(x_1,x_3,...,x_D)$ and rewriting variables 
in the local coordinates, i.e. $\+f(t,\+x) = \+f(t,\+x^w;x_2).$
Assume $\varepsilon$ is a small parameter representing the Knudsen
number, then a typical Knudsen layer ansatz could be
\begin{eqnarray}\label{eq:ant}
	\+f(t,\+x) = \+f_{B}(t,\+x^w;x_2) +
	\varepsilon{\+f}_{K}(t,\+x^w;y);\\ \notag \quad
	\+f_K(t,\+x^w;+\infty)=0;\ \ y=x_2/\varepsilon,
\end{eqnarray}
where $\+f_B$ and $\+f_K$ represent the bulk and  
Knudsen layer solutions respectively with $\+f_K=O(\+f_B)$. 
Formally expand all variables into series on $\varepsilon$,
with $h^{(j)}$ representing the $j$-th order term of $h$, e.g.
\begin{eqnarray}\label{eq:ant2}
		\+f_{B} = \+f_{B}^{(0)} + \varepsilon \+f_{B}^{(1)} +
		o(\varepsilon\+f_B),\quad \+f_K = \+f_K^{(0)} + o(\+f_K).			
\end{eqnarray}
Assume that there are no initial layers and variations
of $\+f_{B}$ are of the same order of $\+f_B$, while variations
of $\+f_K$ are relatively large in the direction normal to
the boundary, i.e.
\[
\pd{\+f_{K}}{s} = O(\+f_K),\ s=t,x_d,\ d\neq 2; \quad  
\pd{\+f_{K}}{y} = \pd{\+f_K}{x_2}\varepsilon = O(\+f_K);
\quad \text{\ when\ } y=O(1).
\]

\renew{Substitute the ansatz into the moment system and match the order
of $\varepsilon$, then we would get equations about $\+f_K^{(0)}$.
If we omit the arguments $t$ and $\+x^w$, rewriting 
$\+f_K^{(0)}(t,\+x^w;y)$ as $\+h(y)$, then we obtain the Knudsen
layer equations (\ref{eq:KL_pre}).
}

\section{Derivation of the boundary condition (\ref{eq:wbc1})}
\label{app:C}
First we make an ansatz that $\displaystyle f=\sum_{|\+\alpha|\leq M}
\omega f_{\+\alpha}\phi_{\+\alpha}$ and $\displaystyle \mathcal{M}^w=
\sum_{|\+\alpha|\leq M} \omega m_{\+\alpha}\phi_{\+\alpha}$, where
the coefficients
\[ f_{\+\alpha}=\ang{f\phi_{\+\alpha}},\quad m_{\+\alpha}=
\ang{\mathcal{M}^w\phi_{\+\alpha}}.\]
Note that we assume the outward normal vector $\+n=(0,-1,0,..,0)$
and $\+u^w\cdot\+n=0.$ So $\+\xi^*=(\xi_1,-\xi_2,\xi_3,..,\xi_D)$ and 
we test the kinetic boundary condition (\ref{eq:Maxwell}) with
$\xi_2\phi_{\+\alpha}(\+\xi)$ to have
\begin{eqnarray*}
	\ang{\xi_2\phi_{\+\alpha}f} &=& \ang{I_{\xi_2>0}
	\xi_2\phi_{\+\alpha}f} + \chi\ang{I_{\xi_2<0}
	\xi_2\phi_{\+\alpha}\mathcal{M}^w} + (1-\chi)\ang
	{I_{\xi_2<0}\xi_2\phi_{\+\alpha}f(\+\xi^*)},
\end{eqnarray*}
where $I_{\xi_2<0}=I_{\xi_2<0}(\+\xi)$ equals one when $\xi_2<0$
and otherwise zero. Substituting the ansatz into the above formula
and noting that $\mathcal{M}^w(\+\xi)=\mathcal{M}^w(\+\xi^*),\ 
\phi_{\+\alpha}(\+\xi)=(-1)^{\alpha_2}\phi_{\+\alpha}(\+\xi^*)$,
we collect the formula into an even-odd parity form:
\[
	\left(1-\frac{\chi}{2}\right)\sum_{\+\beta\in\mathbb{I}_o}\ang
	{\xi_2\omega\phi_{\+\alpha}\phi_{\+\beta}} f_{\+\beta}
	= -\frac{\chi}{2}\sum_{\+\beta\in\mathbb{I}_e}\ang{|\xi_2|\omega
	\phi_{\+\alpha}\phi_{\+\beta}}(f_{\+\beta}-m_{\+\beta}),
	\quad \+\alpha\in\mathbb{I}_e,
\]
where we use the fact that 
$\ang{\xi_2\omega\phi_{\+\alpha}\phi_{\+\beta}}=0$
when $\+\alpha,\+\beta \in\mathbb{I}_e$.
According to the Knudsen layer ansatz, we attach the subscripts
$B$ and $K$ to discriminate the bulk solution and the Knudsen layer
solution. Then we write
\[
	f_{\+\alpha}=f_{\+\alpha,B} + \varepsilon f_{\+\alpha,K},	
\]
where $\varepsilon$ is a small quantity representing the Knudsen
number. Expand all the variables formally into series on
$\varepsilon$, and we assume 
\begin{eqnarray*}
	f_{\+\alpha,B}&=&f_{\+\alpha,B}^{(0)} + \varepsilon 
	f_{\+\alpha,B}^{(1)} + o(\varepsilon f_{\+\alpha}),\\
	f_{\+\alpha,K}&=&f_{\+\alpha,K}^{(0)} + o(f_{\+\alpha}),\\
	m_{\+\alpha}&=&m_{\+\alpha}^{(0)} + \varepsilon 
	m_{\+\alpha}^{(1)} + o(\varepsilon f_{\+\alpha}),
\end{eqnarray*}
where $f_{\+\alpha,B}^{(j)},\ f_{\+\alpha,K}^{(j)}$ and
$m_{\+\alpha}^{(j)}$ have the same magnitude for $j\in\bbN,\
j\geq0.$ When $\chi$ itself is not a small quantity, we can match
the order of $\varepsilon$ in the boundary condition and get
\[
	\left(1-\frac{\chi}{2}\right)\sum_{\+\beta\in\mathbb{I}_o}\ang
	{\xi_2\omega\phi_{\+\alpha}\phi_{\+\beta}} (f_{\+\beta,B}^{(1)}
	+f_{\+\beta,K}^{(0)})
	= -\frac{\chi}{2}\sum_{\+\beta\in\mathbb{I}_e}\ang{|\xi_2|\omega
	\phi_{\+\alpha}\phi_{\+\beta}}(f_{\+\beta,B}^{(1)}
	-m_{\+\beta}^{(1)}+f_{\+\beta,K}^{(0)}).
\]
According to the derivation of Knudsen layer equations,
$f_{\+\beta,K}^{(0)}$ would relate to $\+w$ in
(\ref{eq:KL_p2}), and $m_{\+\beta}$ is calculated in 
Appendix \ref{app:E}. Hence, we get boundary conditions in the 
form of (\ref{eq:wbc1}).

\section{Explicit expressions of (\ref{eq:defS})}
\label{app:B}
We first introduce two differential relations of the one-dimensional
Hermite polynomials (cf. \cite{Fan_new}):
\[
	\pd{}{\xi}\phi_{\alpha+1}(\xi) =
	\sqrt{\alpha+1}\phi_{\alpha}(\xi),\qquad
	\pd{}{\xi}(\omega \phi_{\alpha}) =
	-\sqrt{\alpha+1}\omega\phi_{\alpha+1},
\]
where $\alpha\in\bbN$ and $\xi\in\bbR.$ Due to the orthogonality, we 
only need to calculate
\begin{eqnarray}
	S(\alpha,\beta) &:=& -\sqrt{2\pi}\int_{-\infty}^0\!\!
	\xi\phi_{\alpha}\phi_{\beta}\omega\, \mathrm{d}\xi,\quad
	\alpha,\beta \text{\ are even numbers.} \\ \notag
	&=& -\sqrt{2\pi}\int_{-\infty}^0\!\!
	\left(\sqrt{\alpha}\phi_{\alpha-1}+\sqrt{\alpha+1}
	\phi_{\alpha+1}\right)\phi_{\beta}\omega\, \mathrm{d}\xi.
\end{eqnarray}
Denote by 
\begin{eqnarray*}
I(\alpha,\beta) := \sqrt{2\pi}\int_{-\infty}^0\!\!
	\phi_{\alpha} \phi_{\beta} \omega \mathrm{d}\xi.
\end{eqnarray*}
Integrate by parts using 
$\mathrm{d}\left(\omega\phi_{\beta}\right)=
- \sqrt{\beta+1}\omega\phi_{\beta+1} \mathrm{d}\xi$
or $\mathrm{d}\left(\omega\phi_{\alpha}\right)=
- \sqrt{\alpha+1}\omega\phi_{\alpha+1} \mathrm{d}\xi$,
and we should have two equivalent results:
	\begin{eqnarray*} 
	I(\alpha+1,\beta+1) &=& 
		\left(-\sqrt{2\pi}\phi_{\alpha+1}(0) \phi_{\beta}(0)\omega(0)
		+ \sqrt{\alpha+1}I(\alpha,\beta)\right)/\sqrt{\beta+1} \\ &=& 
	\left(-\sqrt{2\pi}\phi_{\beta+1}(0) \phi_{\alpha}(0)\omega(0)
		+ \sqrt{\beta+1}I(\alpha,\beta)\right)/\sqrt{\alpha+1}.
	\end{eqnarray*}
Noting that $\omega(0)=(2\pi)^{-1/2}$, we denote by
$z_{\alpha}=\phi_{\alpha}(0)$ and have
\begin{eqnarray}
	I(\alpha,\beta) = \left\{\begin{aligned}&\frac{1}{\alpha-\beta}
		\left(\sqrt{\alpha+1}z_{\alpha+1}z_{\beta}
		-\sqrt{\beta+1}z_{\beta+1}z_{\alpha}\right),\quad
		\alpha\neq\beta, \\ &\frac{\sqrt{2\pi}}{2},\quad
	\alpha=\beta,\end{aligned}\right.
\end{eqnarray}
	where $z_0=1,\ z_1=0$ and $z_{n+1}=-\sqrt{n}z_{n-1}/\sqrt{n+1}$.
	Since $z_n=0$ when $n$ is odd, we have
\begin{eqnarray}\notag
	S(\alpha,\beta) &=& -\sqrt{\alpha} I(\alpha-1,\beta) 
	-\sqrt{\alpha+1}I(\alpha+1,\beta)\\
	&=& \label{eq:AAA3}
	\frac{\alpha+\beta+1}{1-(\alpha-\beta)^2}
	z_{\alpha}z_{\beta},\qquad \alpha,\beta \text{\ even}.
\end{eqnarray}

\section{Moments of $\mathcal{M}^w$}
\label{app:E}
We may calculate 
\begin{eqnarray}\label{eq:maJ}
	m_{\boldsymbol{\alpha}} = 
	\int_{\mathbb{R}^3}\frac{\rho^w}{\sqrt{2
	\pi\theta^w}^{3}} \exp\left(-\frac{|\boldsymbol{\xi}- \+u^w|^2}
	{2\theta^w}\right) \phi_{\boldsymbol{\alpha}}
	(\boldsymbol{\xi})\mathrm{d} \boldsymbol{\xi} 
	:= \rho^w \prod_{i=1}^DJ_{\alpha_i}(u_i^w), 
\end{eqnarray}
where $J_{m}(x)$ is a one-dimensional integral defined 
for $m\in \mathbb{N}$ and $x\in \mathbb{R}$ as 
\begin{eqnarray}
J_{m}(x) := \int_{\mathbb{R}}(2\pi\theta^w)^{-\frac{1}{2}}
	\exp\left(-\frac{|\xi-x|^2}{2\theta^w}\right)
\phi_{m}(\xi)\mathrm{d}\xi.
\end{eqnarray}
Integrate by parts using $\mathrm{d}\left(\phi_{m+1}
\right)=\sqrt{m+1}\phi_m \mathrm{d}\xi$, then we have
\begin{eqnarray*}
	J_{m}(x) &=& \frac{1}{\sqrt{m+1}}
\int_{\mathbb{R}}(2\pi\theta^w)^{-\frac{1}{2}}\exp\left(-\frac{|\xi-x|^2}
	{2\theta^w}\right)\phi_{m+1}(\xi)\frac{\xi-x}
	{\theta^w}\mathrm{d}\xi \notag\\
	&=& \frac{1}{\theta^w\sqrt{m+1}}\left(-xJ_{m+1}+\sqrt{m+1} 
	J_m(x)+\sqrt{m+2}J_{m+2}(x)\right),\quad m\geq 0.
\end{eqnarray*} 
This gives the recurrence relation
\begin{equation}\label{eq:C1}
	J_{m}(x) = \frac{1}{\sqrt{m}}\left((\theta^w-1)\sqrt{m-1}
	J_{m-2}(x)+xJ_{m-1}(x)\right),\quad m\geq 2,
\end{equation}
with $J_0(x)=1$ and $J_1(x)=x$ since $\phi_0
(\xi)=1,\ \phi_1(\xi)=\xi$.

If $\mathcal{M}^w$ is linearized around the Maxwellian given
by $\theta_0=1$ and $\+u_0=\+0$, we would have $\theta^w-1=
O(\varepsilon)$ and $u_i^w=O(\varepsilon)$ where $\varepsilon$
is a small quantity. So 
\begin{eqnarray*}
	J_{m}(u_i^w) = o(\varepsilon),\ m>2.
\end{eqnarray*}
Thus, the linearization of $m_{\+\alpha}$ would all be 
higher-order small quantities except the first two-order.

\end{document}